\documentclass[11pt,a4paper,reqno]{amsart}

\usepackage{caption}
\usepackage{subcaption}
\usepackage{amsfonts, amssymb, amscd}
\usepackage{amsmath}
\usepackage{verbatim}
\usepackage{amssymb}
\usepackage{mathrsfs}
\usepackage{mathtools}
\usepackage{graphicx}
\usepackage{cite}
\usepackage[all]{xy}
\usepackage{tikz-cd}

\usepackage{graphicx}
\usepackage{float}
\usepackage[margin=1.5in]{geometry}
\usepackage[pagebackref]{hyperref}
\newtheorem{theorem}{Theorem}[section]
\newtheorem{lemma}[theorem]{Lemma}
\newtheorem{proposition}[theorem]{Proposition}
\newtheorem{corollary}[theorem]{Corollary}

\newtheorem{conjecture}[theorem]{Conjecture}

\numberwithin{equation}{section}

\theoremstyle{definition}
\newtheorem{definition}[theorem]{Definition}

\newtheorem{remark}[theorem]{Remark}

\DeclareMathOperator{\mld}{mld}
\DeclareMathOperator{\lcm}{lcm}

\newcommand{\Z}{{\mathbb Z}}
\newcommand{\Q}{{\mathbb Q}}

\newcommand{\C}{{\mathbb C}}
\newcommand{\A}{{\mathbb A}}

\newcommand{\on}[1]{\operatorname{#1}}

\newcommand{\capt}[3]{\parbox{#1}{\renewcommand{\baselinestretch}{1.0}
                                                           \caption{\label{#2}\small\it #3}}}

\title{Calabi-Yau varieties of large index}

\author{Louis Esser}
\author{Burt Totaro}
\author{Chengxi Wang}

\address{UCLA Mathematics Department, Box 951555, Los Angeles, CA 90095-1555}
\email{esserl@math.ucla.edu}

\email{totaro@math.ucla.edu}

\email{chwang@math.ucla.edu}

\begin{document}

\maketitle

\section{Introduction}

Call a normal projective variety $X$ {\it Calabi-Yau }if its canonical divisor $K_X$ is $\Q$-linearly equivalent to zero. The smallest positive integer $m$
with $mK_X$ linearly equivalent to zero is called the {\it index} of $X$.
A major conjecture on the classification of Calabi-Yau varieties
predicts that, under suitable assumptions on singularities,
the index is bounded in each dimension.
More generally, the following
conjecture predicts an index bound for
Calabi-Yau pairs $(X,D)$; such a pair consists of a normal projective variety
$X$ and an effective $\Q$-divisor $D$ on $X$ such that $K_X + D \sim_{\Q} 0$.
Conjectures in this direction go back to Alexeev
and M\textsuperscript{c}Kernan-Prokhorov,
and Y.~Xu gave a formulation
close to what follows \cite{Alexeev},
\cite[Conjecture 3.8]{MP}, \cite{Xu}.

\begin{conjecture}[Index Conjecture]
\label{indexconj}
Let $n$ be a positive integer and $I \subset [0,1]$ a set of rational numbers
that satisfies the descending chain condition (DCC).  Then there is a positive integer $c(n,I)$ with the following property.
Let $(X,D)$ be a complex klt Calabi-Yau pair of dimension $n$ such that the coefficients of $D$ belong to $I$.
Then $c(n,I)(K_X + D) \sim 0$.
\end{conjecture}

In dimension at most 3, the conjecture is true,
and a similar boundedness statement holds even for slc Calabi-Yau pairs
\cite[Theorem 1.13]{Xu}, \cite[Corollary 1.6]{JL}.  For klt Calabi-Yau
pairs $(X,D)$ with $D \neq 0$, the conjecture also holds in dimension $4$ \cite[Theorem 1.14]{Xu}. An important
result in any dimension
is that all the Calabi-Yau pairs in Conjecture \ref{indexconj}
(with given dimension and coefficients in a given DCC set $I$) actually
have coefficients only in some finite subset of $I$,
by Hacon-M\textsuperscript{c}Kernan-Xu \cite[Theorem 1.5]{HMX}.

For several classes of Calabi-Yau varieties or pairs, we construct
examples which
we conjecture have the largest index in each dimension, with supporting evidence in low dimensions.  In our examples, the index grows doubly exponentially with dimension, and may be written in terms of {\it Sylvester's sequence}, defined recursively by $s_0 = 2$ and $s_n = s_{n-1}(s_{n-1}-1) + 1$ for $n \geq 1$.

First, we produce terminal Calabi-Yau varieties
with large index. The key idea is to apply
mirror symmetry (Remark \ref{mirror})
to our Calabi-Yau varieties with an ample Weil divisor of small volume
\cite[Proposition 3.1]{ETW}.

\begin{theorem}[Corollary \ref{can-term-ex}]
\label{can-term-thm}
For each positive integer $n$, there is a complex terminal Calabi-Yau $n$-fold
with index $(s_{n-1}-1)(2s_{n-1}-3)$. In particular,
this is larger than $2^{2^{n-1}}$.
\end{theorem}

This example should have the largest possible index
among all canonical Calabi-Yau varieties of dimension $n$,
and also among all terminal Calabi-Yau varieties of dimension $n$
(Conjecture \ref{can-term-conj}). For example, our construction
gives a known Calabi-Yau 3-fold of index $66$, the largest possible
\cite[Corollary 5]{MO},
and a new terminal Calabi-Yau
$4$-fold of index $3486$.

Next, we turn to klt Calabi-Yau pairs $(X,D)$
with standard coefficients, meaning that all coefficients of $D$
are of the form $1-\frac{1}{b}$ for positive integers $b$.
Pairs with standard coefficients naturally arise as quotients
by finite groups of varieties $Y$ with $\Q$-Cartier canonical class,
in the sense that $\pi\colon Y\to X$ has $K_Y=\pi^*(K_X+D)$.
For Calabi-Yau pairs, the index can be somewhat bigger
than for varieties. (In fact, the example in Theorem \ref{can-term-thm}
is based on the same construction used to define the pair below.) 
Our example has a simple description
as a pair on a weighted projective space; it can also be viewed
as a hypersurface
in an ill-formed fake weighted projective space.

\begin{theorem}[Theorem \ref{klt-index}]
\label{klt-thm}
For each positive integer $n$, there is a complex klt Calabi-Yau pair
of dimension $n$ with standard coefficients
that has index $(s_n-1)(2s_n-3)$. In particular,
this is larger than $2^{2^{n}}$.
\end{theorem}

This should be the largest possible index among such pairs
(Conjecture \ref{kltpairconj}).

Jihao Liu recently constructed klt Calabi-Yau pairs with
small minimal log discrepancy (mld) \cite[Remark 2.6]{Liu}.
We write out the details as Theorem \ref{smallmld}.
This is related
to several examples by Koll\'ar \cite[Example 8.16]{Kollarpairs},
\cite{Kollarlog}.
By definition, the klt property
for a pair
means that the mld is positive.
It follows from Hacon-M\textsuperscript{c}Kernan-Xu's work
that there is a positive lower bound for the mld's
of all klt Calabi-Yau pairs with standard coefficients
in a given dimension (Proposition \ref{mldbound}).
Liu's example is conjectured to achieve the minimum (Conjecture
\ref{mldconj}).

The pairs in Theorem \ref{klt-thm} and Theorem \ref{smallmld} 
above may be viewed as quotients of certain Calabi-Yau
weighted projective hypersurfaces.  These are also
notable for their extreme topological properties.  We prove that
the total rank of the orbifold cohomology of
these hypersurfaces grows doubly exponentially with dimension.

\begin{theorem}[Theorem \ref{Hodgenumbers}(ii)]
\label{extremecohomology}
For every positive integer $n$, there is a quasi-smooth
Calabi-Yau hypersurface of dimension $n$
whose sum of orbifold Betti numbers is
$$H = 2(s_0-1) \cdots (s_n-1).$$
In particular, this is larger than $2^{2^n}$ for $n > 1$.
\end{theorem}

This should be the largest possible sum of orbifold Betti numbers among 
all projective varieties with quotient singularities
and trivial canonical class
(Conjecture \ref{Hodgesumconj}).

Finally, we conjecture the largest index for a klt Calabi-Yau
variety (rather than a pair) in each dimension (Conjecture
\ref{kltCYconj}). We prove that conjecture in dimension 2:
the largest index of a klt Calabi-Yau surface
is $19$ (Proposition \ref{dim2}).
Also, the smallest mld of a klt Calabi-Yau
surface is $\frac{1}{13}$.

\noindent{\it Acknowledgements. }Esser and Totaro were
supported by NSF grant DMS-2054553.
Thanks to Joaqu\'in Moraga and Sam Payne for useful
conversations.

\section{Notation and background}
\label{notation}

We work over the complex numbers.
For the singularities of the minimal model program,
such as terminal, canonical, Kawamata log terminal (klt),
or log canonical (lc),
some introductions are \cite{Reidyoung, KM}.

For the terminology
of weighted projective spaces such as well-formedness
or quasi-smoothness, see \cite[section 2]{ETW}
for a quick summary, or \cite{Iano-Fletcher}.
A {\it fake weighted projective space }means a projective
toric variety whose divisor class group has rank 1, or equivalently
the projective toric variety associated to a lattice simplex
\cite{Kasprzyk}.
Such a variety is the quotient of a weighted projective space
by a finite abelian group.

Here are the properties of Sylvester's sequence
needed for our examples; a bit more detail is given
in \cite[section 3]{ETW}. The sequence is defined by $s_0=2$
and $s_n=s_{n-1}(s_{n-1}-1)+1$ for $n\geq 1$.
The first few terms are $2$, $3$, $7$, $43$, $1807$. We have $s_n>2^{2^{n-1}}$
for all $n$. Also, $s_n=s_0\cdots s_{n-1}+1$, and hence the numbers
in the sequence are pairwise coprime. Finally, the key
point for our applications is that the sum of the reciprocals
tends very quickly to 1. Namely:
$$\frac{1}{s_0}+\frac{1}{s_1}+\cdots+\frac{1}{s_{n-1}}=1-\frac{1}{s_n-1}.$$

We recall the definition of the minimal log discrepancy (mld).
Let $X$ be a normal variety with a $\Q$-divisor $D$ such that
$K_X+D$ is $\Q$-Cartier. For a proper birational morphism $\mu\colon  X' \rightarrow X$
with $X'$ normal and an irreducible divisor $E \subset X'$, the {\it log discrepancy} of $E$ with respect to $(X,D)$, written $a_E(X,D)$,
is the coefficient of $E$ in the $\Q$-divisor $K_{X'} + E - \mu^*(K_X + D)$.
The {\it center} of $E$ in $X$ is the image $\mu(E)$.  The log discrepancy and the center depend only on the valuation defined by $E$ on the function field of $X$; that allows us to identify some irreducible divisors
on different birational models of $X$.
For a pair $(X,D)$ and a point $x$ of the scheme $X$, the {\it minimal log discrepancy} of $(X,D)$ at $x$ is the infimum
$$\mld_x(X,D) \coloneqq \inf \{a_E(X,D) \colon  \on{center}_X(E) = \overline{x}\},$$
and the (global) {\it minimal log discrepancy} of $(X,D)$ is
$$\mld(X,D) \coloneqq \inf_{x \in X} \mld_x(X,D).$$
Furthermore, when $(X,D)$ is lc (meaning that its mld is nonnegative),
the mld can be computed using the finitely many irreducible
divisors that appear on a log resolution of $(X,D)$
(and so it is rational)
\cite[Definition 7.1]{Kollarsing}.

By definition, a pair $(X,D)$ is klt if and only if its mld
is positive.  Furthermore, in each dimension, there is a positive lower bound for the mld of every klt Calabi-Yau pair with standard coefficients.
More generally:

\begin{proposition}
\label{mldbound}
Let $n$ be a positive integer and $I\subset [0,1]$ a DCC set. Then there is a positive number $\epsilon$ such that every klt Calabi-Yau pair $(X,D)$ of dimension $n$ with coefficients of $D$ in $I$ has mld at least $\epsilon$.
\end{proposition}

\begin{proof}
We repeat the argument of \cite[Lemma 3.13]{CDHJS} for pairs rather than varieties.  Suppose by way of contradiction that there is a sequence of klt Calabi-Yau pairs $(X_i,D_i)$ of dimension $n$ with coefficients in $I$ such that the sequence $\epsilon_i  \coloneqq  \mld(X_i,D_i)$ converges to $0$.  Replacing this with a subsequence, we may suppose that the sequence $\epsilon_i$ is decreasing and that each term is less than 1.  For each positive
integer $i$, define a new klt pair as follows:
choose a point $x_i \in X_i$ with $\mld_{x_i}(X_i,D_i) = \epsilon_i$.  If the closure $Y$ of $x_i$ has codimension 1 in $X_i$,
then $Y$ has coefficient $1-\epsilon_i$ in $D_i$. In this case,
let $(X_i',D_i')$ be $(X_i,D_i)$, i.e., leave the original pair unchanged.
If the point $x_i \in X_i$ has codimension greater than $1$, then choose an exceptional divisor $E_i$ over $X$ with log discrepancy $\epsilon_i$.
Since $\epsilon_i$ is less than 1, \cite[Corollary 1.39]{Kollarsing} gives that
there is a projective birational morphism $\mu\colon X_i' \rightarrow X_i$
for which $X_i'$ is $\Q$-factorial and $E_i$ is the only exceptional divisor.
Let $D_i' = \mu_*^{-1}D_i + (1-\epsilon_i)E_i$.
Then $(X_i',D_i')$ is a new klt pair which is still Calabi-Yau. Indeed,
$$K_{X_i'} + D_i' = \mu^*(K_{X_i}+D_i) \sim_{\Q} 0.$$
Thus, for every
positive integer $i$, we've constructed a klt Calabi-Yau pair
of dimension $n$
that includes the coefficient $1-\epsilon_i$. The coefficients of these pairs are in the union $J$
of the set $I$ with the numbers $1-\epsilon_i$. Since $J$ satisfies
the descending chain condition, this conclusion contradicts
Hacon-M\textsuperscript{c}Kernan-Xu's result \cite[Theorem 1.5]{HMX}.
\end{proof}

For a klt Calabi-Yau pair $(X,D)$ with standard coefficients,
let $m$ be the index of $(X,D)$. Then
the (global) {\it index-1 cover }of $(X,D)$
is a projective variety $Y$ with canonical Gorenstein singularities
such that the canonical class $K_Y$ is linearly equivalent
to zero \cite[Example 2.47, Corollary 2.51]{Kollarsing}.
Here $(X,D)$ is the quotient of $Y$ by an action
of the cyclic group $\mu_m$ such that
$\mu_m$ acts faithfully on $H^0(Y,K_Y)\cong \C$.
(Explicitly, $D$ has coefficient $1-\frac{1}{b}$
on the image of an irreducible divisor on which the subgroup
of $\mu_m$ that acts as the identity has order $b$.)

Finally, we note the existence of equivariant
terminalizations.

\begin{proposition}
\label{equivterm}
Let a finite group $G$ act on a complex variety $X$.
Then there is a $G$-equivariant projective birational
morphism $\pi\colon Y\to X$ with $Y$ terminal and $K_Y$ nef over $X$.
If $X$ is canonical, then $K_Y=\pi^*(K_X)$.
\end{proposition}

\begin{proof}
Using a canonical resolution procedure,
there is a $G$-equivariant resolution of singularities $Z\to X$
\cite[section 3.4.1]{Kollarres}. By the method
of Birkar-Cascini-Hacon-M\textsuperscript{c}Kernan,
we can run a $G$-equivariant minimal model program
for $Z$ over $X$ \cite{BCHM}, \cite[section 4.3]{Prokhorov}.
That gives a $G$-equivariant birational
contraction $Z\dashrightarrow Y$ over $X$ (meaning that $Y \to X$
is a birational morphism) with $Y$ terminal,
$Y$ projective over $X$,
and $K_Y$ nef over $X$. If $X$ is canonical, then
$K_Y=\pi^*(K_X)+\sum a_i E_i$ for some rational numbers $a_i\geq 0$.
Since $K_Y$
is nef over $X$, we have $a_i=0$ for all $i$ by the negativity lemma
\cite[Lemma 3.39]{KM}.
\end{proof}

Without the group action, we can also arrange that $Y$ is $\Q$-factorial,
but that is generally impossible in this equivariant setting.
A variety $Y$ produced by Proposition \ref{equivterm}
is ``$G\Q$-factorial'',
meaning that every $G$-invariant Weil divisor is $\Q$-Cartier;
but that does not imply that $Y$ is $\Q$-factorial
\cite[Example 1.1.1]{Prokhorov}.

\section{Calabi-Yau pairs with small mld or large index}
\label{pairsection}

In this section, we'll construct a Calabi-Yau pair
of large index, Theorem \ref{klt-thm}.
We first present Theorem \ref{smallmld},
Liu's example with small mld \cite[Remark 2.6]{Liu}.

\begin{theorem}
\label{smallmld}
For each positive integer $n$, there is a complex klt Calabi-Yau
pair of dimension $n$
with standard coefficients whose mld is $1/(s_{n+1}-1)$.
In particular, this is less than $1/2^{2^n}$.
\end{theorem}

\begin{proof}
Koll\'ar constructed a klt pair $(X,D)$ with standard coefficients
such that $K_X+D$ is ample and $K_X+D$ has small volume, conjecturally
the smallest in each dimension
\cite{Kollarlog}, \cite[Introduction]{HMXbir}.
A small change gives Liu's Calabi-Yau pair.
Namely, define a pair $(X,D)$ by
\begin{equation}
    \label{mldex}
(X,D) \coloneqq \left( \mathbb{P}^n,\frac{1}{2}H_0+\frac{2}{3}H_1
+\frac{6}{7}H_2+\cdots+\frac{s_n-1}{s_n}H_n
+\frac{s_{n+1}-2}{s_{n+1}-1}H_{n+1}\right),
\end{equation}
where $H_0,\ldots,H_{n+1}$ are $n+2$ general hyperplanes in $\mathbb{P}^n$.
(Here the last coefficient of Koll\'ar's pair has been changed.)
Then $(X,D)$ is klt, it has standard coefficients,
and $K_X+D$ is $\Q$-linearly equivalent to zero.
It is clear from the last coefficient that $(X,D)$
has mld $1/(s_{n+1}-1)$, as we want.
\end{proof}

\begin{conjecture}
\label{mldconj}
Let $(X,D)$ be the pair defined in \eqref{mldex}.  Then the mld
$1/(s_{n+1}-1)$ of $(X,D)$ is the smallest possible
among all klt Calabi-Yau pairs with standard coefficients of dimension $n$.
\end{conjecture}

In dimension $1$, the example above is the klt Calabi-Yau pair
$(\mathbb{P}^1,\frac{1}{2}H_0 + \frac{2}{3}H_1 + \frac{5}{6}H_2)$,
with mld $\frac{1}{6}$,
and in dimension $2$ we have
$(\mathbb{P}^2,\frac{1}{2}H_0 + \frac{2}{3}H_1 + \frac{6}{7}H_2 +
\frac{41}{42}H_3)$, with mld $\frac{1}{42}$. Conjecture \ref{mldconj}
is true in these cases; in dimension $2$, we prove this
as Proposition \ref{42}.

The klt Calabi-Yau pair $(X,D)$ above has index $s_{n+1}-1$,
which is quite good (doubly exponential in $n$);
but we now give a better example for that problem.
Namely, we produce a klt Calabi-Yau pair with standard coefficients
that has conjecturally maximal index.
We first give a simple construction: the properties
are easy to check, but the origin of the example is hidden.
Our pair $(X,D)$ has index $6$ in dimension $1$, $66$ in dimension
$2$, and $3486$ in dimension $3$.

We found this example by applying mirror symmetry
(Remark \ref{mirror}) to our
canonical Calabi-Yau variety of conjecturally minimal volume
\cite[Proposition 3.1]{ETW}. With that approach, our pair
arises as the quotient of a Calabi-Yau hypersurface
in weighted projective space by the action of a cyclic group,
or equivalently as a hypersurface in an ill-formed fake weighted
projective space. We give those descriptions later in this section.

\begin{theorem}
\label{klt-index}
For an integer $n\geq 2$, let $d = 2s_n-2$ and
let $X$ be the weighted projective space 
$\mathbb{P}^n(d^{(n-1)},d-1,1)$, with coordinates $y_1,\ldots,y_{n+1}$.
For $1\leq i\leq n$, let $D_i$ be the divisor $\{y_i=0\}$ on $X$.
Let $D_0$ be the divisor $\{y_1+\cdots+y_{n-1}+y_ny_{n+1}+y_{n+1}^d=0\}$
on $X$. Let 
$$D = \frac{1}{2}D_0 + \frac{2}{3}D_1 + \cdots + \frac{s_{n-1}-1}{s_{n-1}}D_{n-1} + \frac{d-2}{d-1} D_n.$$
Then $(X,D)$ is a klt Calabi-Yau pair of dimension $n$ 
with standard coefficients and with index
$(s_n-1)(2s_n-3)$.

For $n=1$, let $X=\mathbb{P}^1$ with $3$ distinct complex points $p_1,p_2,p_3$.
Then $(X,\frac{1}{2}p_1+\frac{2}{3}p_2+\frac{5}{6}p_3)$
is a klt Calabi-Yau pair with standard coefficients
and with index $(s_1-1)(2s_1-3)=6$.
\end{theorem}

\begin{proof}
The statement for $n=1$ is straightforward, so assume that $n\geq 2$.
Then $X$ is a well-formed weighted projective space.
We have $K_X=O_X(-(n-1)d-(d-1)-1)=O_X(-nd)$. Also, we have linear
equivalences $D_i\sim O_X(d)$
for $0\leq i\leq n-1$, and $D_n\sim O_X(d-1)$. So $K_X+D$ is
$\Q$-linearly equivalent to $O_X(a)$, where
\begin{align*}
a&=-nd+d-2+\sum_{i=0}^{n-1}d\bigg(1-\frac{1}{s_i}\bigg)\\
&=-2+d\bigg(1-\sum_{i=0}^{n-1}\frac{1}{s_i}\bigg)\\
&=-2+\frac{d}{s_n-1}\\
&=0.
\end{align*}
So $(X,D)$ is a Calabi-Yau pair. The coefficients of $D$
are standard.

Since $K_X$ and the divisors $D_i$ are integral multiples
of $O_X(1)$ in the divisor class group of $X$, the index
of $(X,D)$ is the least common multiple
of the denominators of the coefficients of $D$.
Since the numbers in Sylvester's sequence are pairwise coprime,
this lcm is $\lcm(s_n-1,d-1)=\lcm(s_n-1,2s_n-3)=(s_n-1)(2s_n-3)$,
as we want.

It remains to show that $(X,D)$ is klt. Since the klt property is preserved
under taking quotients by finite groups which act freely
in codimension 1
\cite[Corollary 2.43]{Kollarsing}, it suffices to show this
in the coordinate chart $y_i=1$ for each $i$. Equivalently, it suffices
to show that the affine cone $(\A^{n+1},F)$
over $(X,D)$ is klt outside the origin. Here $F$ is a linear combination
of irreducible divisors $F_0,\ldots,F_n$ with coefficients less than 1,
and so it suffices to show that $F_0,\ldots,F_n$ are smooth and transverse
outside the origin. This is clear for $F_1=\{y_1=0\},\ldots,
F_n=\{y_n=0\}$. 

Next, $F_0=\{y_1+\cdots+y_{n-1}+y_ny_{n+1}+y_{n+1}^d=0\}$
in $\A^{n+1}$.
All intersections of subsets of $F_0,\ldots,F_n$
are smooth of the expected dimension except possibly
for $F_0\cap\cdots\cap F_{n-1}$ and $F_0\cap\cdots\cap F_n$.
Here $F_0\cap\cdots\cap F_{n-1}$ (as a scheme) is the curve
$y_ny_{n+1}+y_{n+1}^d=0$ in $\A^2$, which is smooth outside
the origin, as we want. And $F_0\cap\cdots\cap F_n$
is the origin, as a set.
So all intersections are transverse outside the origin in $\A^{n+1}$,
and hence $(X,D)$ is klt.
\end{proof}

This implies Theorem \ref{klt-thm}, using that $s_n>2^{2^{n-1}}$ for all $n$.
We conjecture that this is the example of largest index.

\begin{conjecture}
\label{kltpairconj}
Let $(X,D)$ be the pair in Theorem \ref{klt-index}.  Then the index $(s_n-1)(2s_n-3)$ of $(X,D)$ is the largest possible index among all klt Calabi-Yau pairs of dimension $n$ with standard coefficients.
\end{conjecture}

\begin{proposition}
\label{dim2pair}
Conjecture \ref{kltpairconj} is true in dimensions at most $2$.
Moreover, there is exactly one klt Calabi-Yau pair with standard
coefficients that has index $6$ in dimension $1$ (up to isomorphism),
and there are exactly two
(related by a blow-up) of index $66$ in dimension $2$.
\end{proposition}

\begin{proof}
This is elementary in dimension $1$. (The index-1 cover of $(X,D)$
is the unique elliptic curve over $\C$ with automorphism group of order 6.)
So let $(X,D)$ be a klt Calabi-Yau pair of dimension $2$ with standard
coefficients and index $m$. Let $Y$ be the index-1 cover
of $(X,D)$ (section \ref{notation}). Thus
$Y$ is a projective surface with canonical singularities
such that $K_Y$ is trivial.
Here $(X,D)$ is the quotient of $Y$ by an action
of the cyclic group $\mu_m$ which is purely non-symplectic,
meaning that $\mu_m$ acts faithfully on $H^0(Y,K_Y)\cong \C$.

Let $Z$ be the minimal resolution of $Y$; then $K_Z$ is trivial,
and so $Z$ is an abelian surface or a K3 surface.
The action of $\mu_m$ lifts to $Z$
by Proposition \ref{equivterm}. Clearly the action of $\mu_m$
on $Z$ is still purely non-symplectic. When $Z$ is a K3 surface,
Nikulin showed that $m\leq 66$, and Machida and Oguiso
showed that for $m=66$, $Z$ and the action of $\mu_{66}$ are unique up
to isomorphism and automorphisms of $\mu_{66}$ \cite[Main Theorem 1]{MO}.
When $Z$ is an abelian surface, $\mu_m$ acts faithfully on $H^0(Z,K_Z)\subset
H^2(Z,\C)=\Lambda^2H^1(Z,\C)$, so it acts faithfully on $H^1(Z,\Z)\cong
\Z^4$, which implies that $m\leq 12$.

There remains the problem of classifying the contractions $Z\to Y$
of the K3 surface $Z$ above. Machida and Oguiso show that $Z$ has Picard
rank $2$, with cone of curves spanned by one $(-2)$-curve $C$ and one elliptic
curve with self-intersection zero \cite[section 2]{MO}.
It follows that $C$ is the only $(-2)$-curve in $Z$.
So the only canonical K3 surfaces with minimal resolution $Z$
are $Z$ itself and the surface $Y$ with a node obtained by contracting
$C$. So a klt Calabi-Yau pair with standard coefficients, dimension $2$,
and index $66$ is isomorphic to $Z/\mu_{66}$ or to $Y/\mu_{66}$. The surface $X$
of Theorem \ref{klt-index} is the latter one, since its index-1
cover (Proposition \ref{wpsindex}) has a node.
\end{proof}

We now describe the example of Theorem \ref{klt-index}
as a quotient of a hypersurface in weighted projective space.

\begin{proposition}
\label{wpsindex}
For a positive integer $n$, let $d \coloneqq 2s_n - 2 = 2s_0 \cdots s_{n-1}$.  Then the hypersurface $X'$ of degree $d$ in the weighted projective space
$Y'  \coloneqq  \mathbb{P}(d/s_0,\ldots,d/s_{n-1},1,1)$ defined by the equation $x_0^2 + x_1^3 + \cdots + x_{n-1}^{s_{n-1}} + x_n^{d-1} x_{n+1} + x_{n+1}^d = 0$
is quasi-smooth of dimension $n$,
canonical, and has $K_{X'}$ linearly equivalent
to zero.
\end{proposition}

\begin{proof}
Since the weighted projective space $Y'$ has two weights equal to 1, it is well-formed. One can verify that the hypersurface $X'$ is quasi-smooth, meaning that
its affine cone is smooth outside the origin
in $\mathbb{A}^{n+2}$.  The sum of the weights is $d/s_0 + \cdots + d/s_{n-1} + 2 = d(1-1/(s_n-1)) + 2 = d$,
and so $K_{X'}$ is linearly equivalent to zero. Since $X'$ is quasi-smooth,
it has quotient singularities, hence is klt. Since $K_{X'}$
is linearly equivalent to zero, it is Cartier. It follows
that $X'$ is canonical.
\end{proof}

\begin{proposition}
\label{largeindex}
With the same notation as above, let $m \coloneqq (s_n-1)(2s_n - 3)$ and let the group
$\mu_m$ of $m$th roots of unity act on the weighted projective space $Y'$ as follows. For any $\zeta \in \mu_m$,
\begin{equation}
\label{groupaction}
    \zeta[x_0 : \cdots  : x_{n+1}] = [\zeta^{d/(2s_0)} x_0: \zeta^{d/(2s_1)} x_1 : \cdots : \zeta^{d/(2s_{n-1})} x_{n-1} : x_n : \zeta^{d/2} x_{n+1}].
\end{equation}
The hypersurface $X'$ is invariant under this group action. The quotient of $X'$ by $\mu_m$ is a klt Calabi-Yau pair $(X,D)$ with standard coefficients and index $m$, isomorphic to the pair in Theorem \ref{klt-index}.
\end{proposition}

We omit the proof, since this is just a different description
of the same example.

\begin{remark}
\label{mirror}
\hspace{2em}
\begin{enumerate}
\item The index $m \coloneqq (s_n-1)(2s_n-3) = (d-1)d/2$ in this example
coincides with the degree of the hypersurface which we conjecture is the canonical Calabi-Yau $n$-fold with an ample Weil divisor of minimum volume \cite[Conjecture 1.2]{ETW}.  Indeed, the hypersurface defined in Proposition \ref{wpsindex} is related to the small volume example by mirror symmetry.
Specifically, the hypersurface $X'$ in Proposition \ref{wpsindex} above is the Berglund-H\"ubsch-Krawitz (BHK) mirror of the hypersurface
$$\widehat{X'_m} \subset \mathbb{P}(m/s_0,\ldots,m/s_{n-1},s_n-1,s_n-2)$$
of degree $m$ defined by the equation $x_0^2 + x_1^3 + \cdots + x_{n-1}^{s_n-1} + x_n^{d-1} + x_n x_{n+1}^{d} = 0$. That is, the equations for $X'$ and $\widehat{X'}$, each with
exactly $n+2$ monomials in $n+2$ variables, are the transposes
of each other, in terms of the explicit BHK recipe
for mirror symmetry \cite[Section 2.3]{ABS}.
\item Section \ref{cohosection} describes the extreme topological behavior
of these Calabi-Yau hypersurfaces.
\item Another striking feature of Proposition \ref{wpsindex} is
that the weighted projective space $\mathbb{P}(d/s_0,\ldots,d/s_{n-1},1,1)$ has the largest anticanonical volume of any canonical toric Fano $(n+1)$-fold \cite[Corollary 1.3]{BKN}.
\end{enumerate}
\end{remark}

For possible future use, we give one last description of
the klt Calabi-Yau pair $(X,D)$ of Theorem \ref{klt-index},
now as a hypersurface in an ill-formed
fake weighted projective space. (Ill-formedness means that we have a toric
pair, not just a toric variety.)
It is convenient to use the language
of toric Deligne-Mumford stacks \cite{BCS}.  Just as there is a projective $\Q$-factorial toric variety associated to a complete simplicial fan in a lattice $N$, we can associate a toric Deligne-Mumford stack to the data of a {\it stacky fan} $\mathbf{\Sigma}$.  A stacky fan $\mathbf{\Sigma}$ is a triple $(N,\Sigma, \beta)$, where $N$ is a finitely generated abelian group, $\Sigma$ is a complete simplicial fan in $N_{\Q} = N \otimes_{\Z} \Q$ with $r$ rays, and $\beta$
is a set of $r$ elements of $N$ that span the rays of $\Sigma$.
(They need not be the smallest lattice points on these rays.)
In our case, $N \cong \Z^{n+1}$ is torsion-free, which means that the
associated stack has trivial generic stabilizer. If the chosen point
of a given ray $\rho$
is a positive integer $b$ times the smallest lattice point, then the associated
toric stack has stabilizer group $\mu_b$ along the irreducible divisor $D$
associated to $\rho$, and the associated pair has coefficient $(b-1)/b$
along $D$.

Define a stacky fan $\mathbf{\Sigma}$ using the following distinguished points generating rays of a simplicial fan in $N \cong \Z^{n+1}$, where $e_0,\ldots,e_n$ is a basis for $N$:
$$v_i =
\begin{cases}
s_i e_i, & i = 0,\ldots, n-1, \\
(d-1)e_n, & i = n, \\
-d(e_0 + \cdots + e_{n-1}) - (d-1)e_n, & i = n+1.
\end{cases}
$$

Most of these lattice points are not primitive on the respective rays, hence the stacky behavior.  Note also that $(d/s_0)v_0 + \cdots (d/s_{n-1})v_{n-1} + v_n + v_{n+1} = 0$.  Therefore, these vectors define the weighted projective space $\mathbb{P}(d/s_0,\ldots, d/s_{n-1},1,1)$ in the lattice spanned by $v_0,\ldots,v_{n+1}$.  The resulting stack will thus be a quotient of this weighted projective space, namely the quotient in Proposition \ref{largeindex}.
The pair $(X,D)$ is the pair associated to a hypersurface in this stack.

\section{Terminal Calabi-Yau varieties of large index}

Building on the klt Calabi-Yau pair of Theorem \ref{klt-index},
we now construct terminal
Calabi-Yau varieties with large index, conjecturally optimal.

\begin{corollary}
\label{can-term-ex}
Let $n$ be an integer at least $2$. Let $X'$ be the Calabi-Yau hypersurface of dimension $n-1$ with an action of the cyclic group of order
$m \coloneqq (s_{n-1}-1)(2s_{n-1}-3)$ from Proposition \ref{largeindex}.
Let $Z$ be a $\mu_m$-equivariant terminalization of $X'$
(Proposition \ref{equivterm}). Let $E$ be a smooth elliptic curve,
and let $p\in E(\C)$ be a point of order $m$.
Let $\mu_m$ act on $Z \times E$ by
$$\zeta (x,y) = (\zeta(x), y + p),$$
where $\zeta$ is a primitive $m$th root of unity.
Then the quotient $S = (Z\times E)/\mu_m$
is a terminal Calabi-Yau $n$-fold
that has index $m$.
\end{corollary}

\begin{proof}
The variety $X'$ is canonical, by Proposition \ref{wpsindex}.
So $K_Z=\pi^*(K_{X'})$ by Proposition \ref{equivterm},
and hence $Z$ is a terminal projective
variety with $K_Z$ trivial.
Since $E$ is a complex elliptic curve, it has a point $p$ of order $m$.
Translation by an non-identity point has no fixed points on $E$, so $\mu_m$ acts freely on $E$. It follows
that $\mu_m$ acts freely on $Z\times E$, and so the quotient
$S=(Z\times E)/\mu_m$ is a Calabi-Yau variety (rather than a pair).
Since $Z$ is terminal
and $\mu_m$ acts freely on $Z\times E$,
the quotient $S=(Z\times E)/\mu_m$ is also terminal.

It remains to see that $S$ has index $m$.  The group $\mu_m$ acts trivially on $H^0(E,K_E)$ since global sections of $K_E$ are translation-invariant.  Therefore, an element of $\mu_m$ acts trivially on $H^0(Z \times E, K_{Z \times E})
=H^0(Z,K_{Z})\otimes_{\C} H^0(E,K_E)$ if and only if it acts trivially on $H^0(Z,K_{Z})$.  So the index of $S$ is equal to the index of the quotient stack $[Z/\mu_m]$
(or the associated klt pair), namely $m$.
\end{proof}

We conjecture that this example is optimal.

\begin{conjecture}
\label{can-term-conj}
Let $n$ be an integer at least $2$ and let $X$ be the Calabi-Yau variety of Corollary \ref{can-term-ex}.  Then the index $m = (s_{n-1}-1)(2s_{n-1}-3)$ of $X$ is the largest possible index among all terminal Calabi-Yau varieties of dimension $n$, and also among all canonical Calabi-Yau varieties
of dimension $n$.
\end{conjecture}

This conjecture holds in dimensions at most $3$.  Indeed, in dimension $2$, terminal singularities are smooth, and the largest possible
index of a smooth Calabi-Yau surface is $6 = (s_1 - 1)(2s_1 - 3)$,
which occurs for a ``bielliptic'' surface $(W\times E)/\mu_6$
as above \cite[Corollary VIII.7]{Beauville}.
Likewise, the largest possible index of a canonical or terminal
Calabi-Yau 3-fold is $66 = (s_2-1)(2s_2-3)$ \cite[Corollary 5]{MO}.

For $n$ at most $3$, there is in fact a {\it smooth }Calabi-Yau
$n$-fold with index $(s_{n-1}-1)(2s_{n-1}-3)$
\cite[Corollary 5]{MO}.
To see this in dimension $3$, let $X'$ be the canonical Calabi-Yau surface
from Proposition \ref{largeindex} with an action of $\mu_{66}$.
Let $Z\to X'$ be a $\mu_{66}$-equivariant terminalization
of $X'$ (Proposition \ref{equivterm}). In dimension $2$, $Z$ is smooth
and unique, known as the minimal resolution of $X'$.
Also, $K_Z=\pi^*(K_{X'})$, and so $K_Z$ is trivial (more precisely,
$Z$ is a smooth K3 surface).
Then the quotient $(Z\times E)/\mu_{66}$ as in Corollary
\ref{can-term-ex} is a smooth Calabi-Yau $3$-fold of index $66$.
It would be interesting to construct smooth Calabi-Yau varieties
of large index in high dimensions.

\section{Calabi-Yau varieties with large Betti numbers}
\label{cohosection}

The Calabi-Yau hypersurfaces in section \ref{pairsection}
are also notable for their topological properties.
In dimension $3$, the family
of hypersurfaces $X'_{84}\subset
\mathbb{P}^4(42,28,12,1,1)$ of Proposition \ref{wpsindex}
is an extreme in Kreuzer and Skarke's famous list
of Calabi-Yau 3-folds \cite[section 3]{KS}. (They found
more than $400$ million canonical 3-folds with trivial canonical class.)
In particular, a crepant
resolution of $X'$ has the most negative known Euler characteristic, $-960$.
(The special hypersurface
in Proposition \ref{wpsindex} with an action of
$\mu_{3486}$ seems to be new, though.)
The BHK mirror to the special hypersurface $X'$
is the small-volume example
$$\widehat{X'_{3486}}\subset \mathbb{P}^4(1743,1162,498,42,41).$$
That has the largest known Euler
characteristic among $K$-trivial 3-folds, $960$ (for a crepant resolution).

A third hypersurface, $X'_{1806} \subset \mathbb{P}^4(903,602,258,42,1)$,
has the same sum of Betti numbers (of a crepant resolution)
as in the previous two examples, namely $1008$. All three of these hypersurfaces 
are visible at the top of
the graph of all pairs of Hodge numbers from Kreuzer and Skarke's list,
shown in Figure \ref{BasicKSplot} \cite{Candelas}.  
\begin{figure}[!t]
\begin{center}
\includegraphics[width=3.0in]{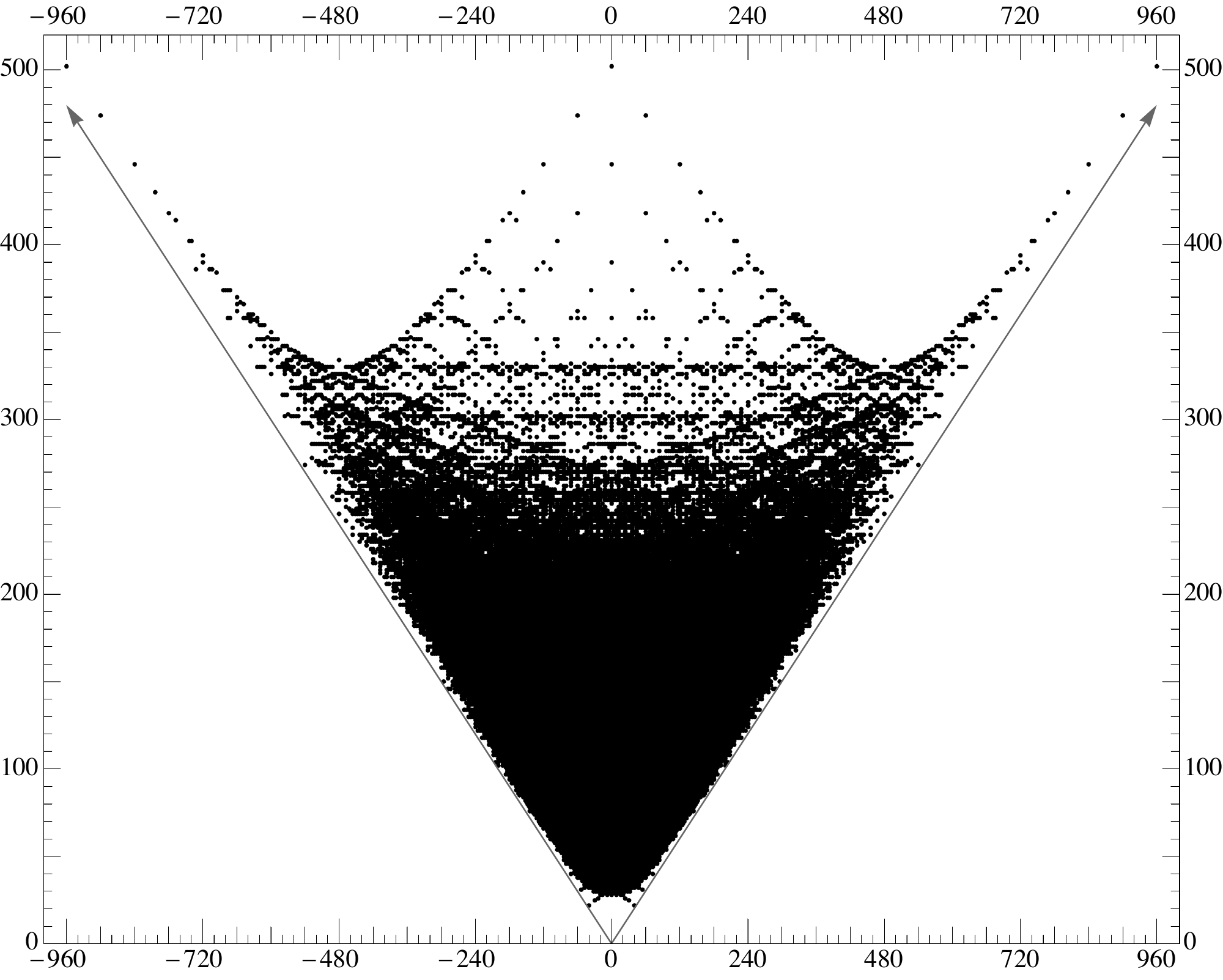} 
\vskip0pt
\capt{6.0in}{BasicKSplot}{Hodge numbers of the Calabi-Yau 3-folds found by Kreuzer and Skarke; $x$-axis is $\chi=2(h^{1,1}-h^{2,1})$, $y$-axis is $h^{1,1}+h^{2,1}$.}
\end{center}
\vspace*{-15pt}
\end{figure}
In this section, we'll show that these 
observations in dimension $3$ fit into a pattern that holds
for all dimensions $n$. 

In higher dimensions, we do not know whether our weighted projective
hypersurfaces have crepant resolutions. Instead, we'll 
consider Chen-Ruan's {\it orbifold cohomology} for these hypersurfaces
\cite{CR}. For a projective variety with Gorenstein quotient singularities,
orbifold cohomology (as a Hodge structure) agrees
with the cohomology of a crepant
resolution if one exists \cite[Corollary 1.5]{Yasuda}.
Assuming Gorenstein quotient singularities,
the orbifold Hodge numbers also agree with
Batyrev's {\it stringy Hodge numbers }\cite[Remark 1.4]{Yasuda}.

We'll make use of the following formula for orbifold Hodge numbers,
building on Vafa's ideas
(see \cite[(4.2)]{Batyrev}).  Given a quasi-smooth hypersurface $X$ of degree $d$ in 
$\mathbb{P}(a_0,\ldots,a_{n+1})$, where $a_0 + \cdots + a_{n+1} = d$, we write: $q_i \coloneqq a_i/d$ and $\tilde{\theta}_i(\ell)$ 
is the fractional part of $\ell a_i/d$.  Then the orbifold Hodge number $h_{\mathrm{orb}}^{p,q}$ 
of the mirror of $X$ is the coefficient of $t^p \bar{t}^q$ in the expression

\begin{equation}
\label{Hodgeformula}
P(t,\bar{t}) = \sum_{\ell = 0}^{d-1}\left[ \prod_{\tilde{\theta}_i(\ell) = 0} \frac{1 - (t \bar{t})^{1-q_i}}{1-(t \bar{t})^{q_i}} \prod_{\tilde{\theta}_i(\ell) \neq 0} (t \bar{t})^{\frac{1}{2}-q_i}\left(\frac{t}{\bar{t}}\right)^{\tilde{\theta}_i(\ell) - \frac{1}{2}}\right].
\end{equation}

In \cite{Batyrev}, $P(t,\bar{t})$ is defined as the part of 
the above expression with integer exponents, but it will be
useful to keep track of the parts with fractional exponents as well.

For each dimension $n$, there are three Calabi-Yau 
hypersurfaces which we expect to have the ``most extreme cohomology":
\begin{align*}
    X_1^{(n)} & \coloneqq \{x_0^{s_0} + x_1^{s_1} + \cdots + x_{n-1}^{s_{n-1}} + x_n^{d-1} x_{n+1} + x_{n+1}^d = 0\} \\ 
     & \subset \mathbb{P}(d/s_0,\ldots,d/s_{n-1},1,1) \text{ of degree } d = 2s_{n}-2,  \\
    X_2^{(n)} & \coloneqq \{x_0^{s_0} + \cdots + x_n^{s_n} + x_{n+1}^d = 0\} \\
     & \subset \mathbb{P}(d/s_0,\ldots,d/s_n,1) \text{ of degree } d = s_{n+1}-1, \text{ and} \\
    X_3^{(n)} & \coloneqq \{x_0^{s_0} + x_1^{s_1} + \cdots + x_{n-1}^{s_{n-1}} + x_n^{2s_n-3} + x_n x_{n+1}^{2s_n-2} = 0\} 
    \\ & \subset \mathbb{P}(d/s_0,\ldots,d/s_{n-1},s_n-1,s_n-2) \text{ of degree } d = (s_n-1)(2s_n-3).
\end{align*}
Note that any quasi-smooth hypersurface with the same degree and weights as one of
the examples above would have identical cohomology, but we pick specific 
hypersurfaces to highlight the connection to previous sections.
The first and third examples form a mirror pair and are the source of our
large index and small volume results, respectively (see Remark \ref{mirror}).
Hypersurfaces of degree $d = s_{n+1}-1$ in $\mathbb{P}(d/s_0,\ldots,d/s_n,1)$
first appeared in \cite[Lemma 3.10]{ETW}.  The specific Fermat hypersurface
$X_2^{(n)}$ is self-mirror and can be viewed as the source of our conjecturally optimal small mld 
example.  Indeed, the pair $(X,D)$ in \ref{mldex} is the 
quotient of this hypersurface by an action of $\mu_{s_{n+1}-1}$.

\begin{theorem}
\label{Hodgenumbers}
Let $n$ be a positive integer. Then the following properties hold:
\begin{enumerate}
    \item[(i)] $h^{p,q}_{\mathrm{orb}}(X_i^{(n)}) = 0$ for $i = 1,2,3$ unless $p = q$ or $p + q = n$.
    \item[(ii)] The sum of the orbifold Betti numbers of $X_i^{(n)}$ for $i = 1,2,3$ is 
    $$H = 2(s_0-1) \cdots (s_n-1).$$ 
    This equals the orbifold Euler characteristic of $X$ if $n$ is even.
    \item[(iii)] Suppose that $n$ is odd.  Then the dimension of the middle orbifold cohomology group is
    $$\dim(H_{\mathrm{orb}}^n(X_i^{(n)},\Q)) = 
    \begin{cases}
    (s_0 - 1) \cdots (s_{n-1}-1)(2s_n-4), & i = 1, \\
    (s_0-1) \cdots (s_n-1), & i = 2, \\
    2(s_0 - 1) \cdots (s_{n-1}-1), & i = 3.
    \end{cases} $$
    The orbifold Euler characteristic is 
    $$\chi_{\mathrm{orb}}(X_i^{(n)}) = 
     \begin{cases}
    -(s_0 - 1) \cdots (s_{n-1}-1)(2s_n-6), & i = 1, \\
    0, & i = 2, \\
    (s_0 - 1) \cdots (s_{n-1}-1)(2s_n-6), & i = 3.
    \end{cases} 
    $$
\end{enumerate}
\end{theorem}

\begin{remark}
The proof will actually show a slight generalization of (i): the stated vanishing of orbifold Hodge numbers holds for any quasi-smooth Calabi-Yau weighted hypersurface $X_d \subset \mathbb{P}(a_0,\ldots,a_{n+1})$ satisfying the conditions that a) each $a_i$ divides $d$ and b) the quotients $d/a_i$ that are strictly smaller than $d$ are pairwise coprime.  Mirror symmetry preserves the vanishing, so it will also hold for mirrors of such hypersurfaces, e.g., $X_3^{(n)}$. However, it does not hold in general.  A simple counterexample is the fourfold $X_{12} \subset \mathbb{P}(3^{(3)},1^{(3)})$, which has $h^{1,2} = 3$.
\end{remark}

The sum $H = 1008$ of Betti numbers common to our three examples in dimension $3$ is the largest known among all Calabi-Yau threefolds (a crepant resolution always exists in this case). The situation in dimension $3$ motivates the following conjecture:

\begin{conjecture}
\label{Hodgesumconj}
In dimension $n$,
the largest possible sum of orbifold Betti numbers for a projective
variety with quotient singularities and trivial canonical class is
$$H = 2(s_0-1) \cdots (s_n-1).$$  
For odd $n$, the smallest possible orbifold Euler characteristic is 
$-(s_0-1) \cdots (s_{n-1}-1)(2s_n-6)$, and the largest is $(s_0-1) \cdots (s_{n-1}-1)(2s_n-6)$.
\end{conjecture}

The individual orbifold Hodge numbers are harder to compute, but the values for our examples in dimensions $3$ and $4$ are shown in Figure \ref{Hodgenum3} and Figure \ref{Hodgenum4}, respectively.

\begin{figure}
\centering
\begin{subfigure}{.33\textwidth}
  \centering
{\small \[
\setlength{\arraycolsep}{0.3em}
\begin{array}{ccccccccc}
&&&&1&&&& \\[2mm]
&&&0&&0&&& \\[2mm]
&& 0&& 11 && 0 && \\ [2mm]
& 1 && 491 && 491 && 1 & \\[2mm]
&& 0&& 11 && 0 && \\[2mm]
&&&0&&0&&&\\[2mm]
&&&&1&&&& \\[2mm]
\end{array}
\]}
  \caption{$X_1^{(3)}$}
  \label{Hodgenum3:1}
\end{subfigure}%
\begin{subfigure}{.33\textwidth}
  \centering
{\small \[
\setlength{\arraycolsep}{0.3em}
\begin{array}{ccccccccc}
&&&&1&&&& \\[2mm]
&&&0&&0&&& \\[2mm]
&& 0&& 251 && 0 && \\ [2mm]
& 1 && 251 && 251 && 1 & \\[2mm]
&& 0&& 251 && 0 && \\[2mm]
&&&0&&0&&&\\[2mm]
&&&&1&&&& \\[2mm]
\end{array}
\]}
  \caption{$X_2^{(3)}$}
  \label{Hodgenum3:2}
\end{subfigure}
\begin{subfigure}{.33\textwidth}
  \centering
{\small \[
\setlength{\arraycolsep}{0.3em}
\begin{array}{ccccccccc}
&&&&1&&&& \\[2mm]
&&&0&&0&&& \\[2mm]
&& 0&& 491 && 0 && \\ [2mm]
& 1 && 11 && 11 && 1 & \\[2mm]
&& 0&& 491 && 0 && \\[2mm]
&&&0&&0&&&\\[2mm]
&&&&1&&&& \\[2mm]
\end{array}
\]}
  \caption{$X_3^{(3)}$}
  \label{Hodgenum3:3}
\end{subfigure}
\caption{The orbifold Hodge diamonds of three extreme Calabi-Yau $3$-folds.}
\label{Hodgenum3}
\end{figure}

\begin{figure}
\centering
\begin{subfigure}{.55 \textwidth}
  \centering
{\small \[
\setlength{\arraycolsep}{0.4em}
\begin{array}{ccccccccccc}
&&&&&1&&&&& \\[2mm]
&&&&0&&0&&&& \\[2mm]
&&& 0&& 252 && 0 &&& \\ [2mm]
&& 0 && 0 && 0 && 0 && \\[2mm]
& 1 && 303148 && 1213644 && 303148 && 1 & \\ [2mm]
&& 0 && 0 && 0 && 0 && \\ [2mm]
&&& 0&& 252 && 0 &&& \\[2mm]
&&&&0&&0&&&&\\[2mm]
&&&&&1&&&&&\\[2mm]
\end{array}
\]}
  \caption{$X_1^{(4)}$}
  \label{Hodgenum4:1}
\end{subfigure}%
\begin{subfigure}{.45\textwidth}
  \centering
{\small \[
\setlength{\arraycolsep}{0.4em}
\begin{array}{ccccccccccc}
&&&&&1&&&&& \\[2mm]
&&&&0&&0&&&& \\[2mm]
&&& 0&& 151700 && 0 &&& \\ [2mm]
&& 0 && 0 && 0 && 0 && \\[2mm]
& 1 && 151700 && 1213644 && 151700 && 1 & \\ [2mm]
&& 0 && 0 && 0 && 0 && \\ [2mm]
&&& 0&& 151700 && 0 &&& \\[2mm]
&&&&0&&0&&&&\\[2mm]
&&&&&1&&&&&\\[2mm]
\end{array}
\]}
  \caption{$X_2^{(4)}$}
  \label{Hodgenum4:2}
\end{subfigure}

\begin{subfigure}{.5\textwidth}
  \centering
{\small \[
\setlength{\arraycolsep}{0.4em}
\begin{array}{ccccccccccc}
&&&&&1&&&&& \\[2mm]
&&&&0&&0&&&& \\[2mm]
&&& 0&& 303148 && 0 &&& \\ [2mm]
&& 0 && 0 && 0 && 0 && \\[2mm]
& 1 && 252 && 1213644 && 252 && 1 & \\ [2mm]
&& 0 && 0 && 0 && 0 && \\ [2mm]
&&& 0&& 303148 && 0 &&& \\[2mm]
&&&&0&&0&&&&\\[2mm]
&&&&&1&&&&&\\[2mm]
\end{array}
\]}
  \caption{$X_3^{(4)}$}
  \label{Hodgenum4:3}
\end{subfigure}
\caption{The orbifold Hodge diamonds of three extreme Calabi-Yau $4$-folds.}
\label{Hodgenum4}
\end{figure}

\begin{proof}[Proof of Theorem \ref{Hodgenumbers}]
The properties of $X_1^{(n)}$ will imply those for $X_3^{(n)}$ by mirror symmetry  \cite[Theorem 4]{ChiodoRuan}; so we'll focus on the first two examples.  For the moment, let $X$ denote any quasi-smooth hypersurface with the property that all weights $a_i$ divide the degree $d$.

First, we'll make the substitutions $t = x^d$, $\bar{t} = y^d$ to eliminate fractional exponents in \eqref{Hodgeformula} and define a polynomial
$$Q(x,y)  \coloneqq  P(x^d,y^d) = \sum_{\ell = 0}^{d-1}\left[ \prod_{\tilde{\theta}_i(\ell) = 0} \frac{1 - (xy)^{d-a_i}}{1-(xy)^{a_i}} \prod_{\tilde{\theta}_i(\ell) \neq 0} (xy)^{\frac{d}{2}-a_i}\left(\frac{x}{y}\right)^{d \tilde{\theta}_i(\ell) - \frac{d}{2}}\right].$$
Write
\begin{equation}
\label{ellterm}
    Q_{\ell}(x,y)  \coloneqq  \prod_{\tilde{\theta}_i(\ell) = 0} \frac{1 - (xy)^{d-a_i}}{1-(xy)^{a_i}} \prod_{\tilde{\theta}_i(\ell) \neq 0} (xy)^{\frac{d}{2}-a_i}\left(\frac{x}{y}\right)^{d \tilde{\theta}_i(\ell) - \frac{d}{2}}
\end{equation}
for the $\ell$th term in the sum.

Since each $a_i$ divides $d$, we have that
$$\frac{1 - (xy)^{d-a_i}}{1-(xy)^{a_i}} = 1 + (xy)^{a_i} + \cdots + (xy)^{d-2a_i}.$$
The orbifold Hodge numbers of $X$ correspond to coefficients of monomials in $Q(x,y)$ of the form $x^a y^b$ with both $a$ and $b$ divisible by $d$.  First, we claim that for any monomial $x^a y^b$ with a nonzero coefficient in $Q(x,y)$, $d$ divides $a$ if and only if $d$ divides $b$.  This is because $Q(x,y)$ is a polynomial in $xy$ with the exception of the term $\left(\frac{x}{y}\right)^{d \tilde{\theta}_i(\ell) - \frac{d}{2}}$.  The difference of the exponents of $x$ and $y$ in this expression is $2d \tilde{\theta}_i(\ell) - d$.  But for any fixed $\ell$, we have
$$\sum_{i,\tilde{\theta}_i(\ell) \neq 0} (2d \tilde{\theta}_i(\ell) - d) \equiv \sum_{i,d \nmid \ell a_i} (2 \ell a_i -d) \equiv \ell \sum_i 2 a_i \equiv 0 \bmod d. $$
Here the second-to-last congruence holds because adding the terms $2 \ell  a_i$ for which $d | \ell a_i$ does not change the value modulo $d$ (removing extraneous $d$ terms also, of course, does not change this value).
Therefore, each summand $Q_{\ell}(x,y)$ is a polynomial in $xy$ multiplied by a monomial with the powers of $x$ and $y$ differing by a multiple of $d$.  This proves the claim.  In particular, the sum of all orbifold Hodge numbers is the same as the sum of coefficients of powers of $x^d$ in 
$$q(x)  \coloneqq  Q(x,1) = \sum_{\ell = 0}^{d-1}\left[ \prod_{\tilde{\theta}_i(\ell) = 0} (1 + x^{a_i} + \cdots + x^{d-2a_i}) \prod_{\tilde{\theta}_i(\ell) \neq 0} x^{d \tilde{\theta}_i(\ell)-a_i}\right].$$

To isolate powers of $x^d$, we can sum the values of this polynomial over $d$th roots of unity.  More precisely, let $H$ be the sum of Hodge numbers and $\zeta$ a primitive $d$th root of unity.  Then:
$$H = \frac{1}{d}\sum_{j = 0}^{d-1} q(\zeta^j) = \sum_{j = 0}^{d-1} \sum_{\ell = 0}^{d-1}\left[ \prod_{\tilde{\theta}_i(\ell) = 0} (1 + \zeta^{ja_i} + \cdots + \zeta^{j(d-2a_i)}) \prod_{\tilde{\theta}_i(\ell) \neq 0} \zeta^{j(d \tilde{\theta}_i(\ell)-a_i)}\right].$$

We may simplify this further by noticing, as before, that the sum of all the $d \tilde{\theta}_i(\ell)$ is divisible by $d$ for fixed $\ell$.  We can therefore drop the $d \tilde{\theta}_i(\ell)$ term in the last exponent. Further, since the sum of the $a_i$ is $d$, we have
$$\prod_{\tilde{\theta}_i(\ell) \neq 0} \zeta^{j(-a_i)} = \prod_{\tilde{\theta}_i(\ell) = 0} \zeta^{ja_i}.$$

After making this substitution, both products are now indexed over the weights $a_i$ for which $\tilde{\theta}_i(\ell) = 0$ (i.e., for which $d | \ell a_i)$.  Thus, we may combine them into one:
$$H = \frac{1}{d}\sum_{j = 0}^{d-1} \sum_{\ell = 0}^{d-1} \prod_{i, d | \ell a_i}(\zeta^{ja_i} + \zeta^{j(2a_i)} + \cdots + \zeta^{j(d-a_i)}).$$

The quantity $\zeta^{ja_i} + \zeta^{j(2a_i)} + \cdots + \zeta^{j(d-a_i)}$ equals $-1$ if $d \nmid ja_i$ and equals $\frac{d}{a_i} - 1$ if $d | ja_i$, since all terms will be equal to $1$ in this case.  We'll switch sums and write this as
\begin{equation*}
    H = \frac{1}{d}\sum_{\ell = 0}^{d-1} \sum_{j = 0}^{d-1} \prod_{i,d|\ell a_i}
\begin{cases}
\frac{d}{a_i} - 1, & d|ja_i, \\
-1, & d \nmid ja_i
\end{cases}.
\end{equation*}
Given a positive integer $c$, we'll use the notation
$$f_c(j) \coloneqq \begin{cases}
c - 1, & c|j, \\
-1, & c \nmid j
\end{cases}.$$
In this notation, 
\begin{equation}
\label{Hodgesum}
    H = \frac{1}{d}\sum_{\ell = 0}^{d-1} \sum_{j = 0}^{d-1} \prod_{i,d|\ell a_i}
f_{\frac{d}{a_i}}(j).
\end{equation}
The same reasoning shows that the contribution to $H$ coming from $Q_{\ell}(x,y)$ is $\frac{1}{d}\sum_{j = 0}^{d-1} \prod_{i,d|\ell a_i} f_{\frac{d}{a_i}}(j)$.  The following lemma will help to evaluate such sums.

\begin{lemma}
\label{counting}
Let $C$ be a nonempty finite set of pairwise coprime positive integers and $d$ a positive integer divisible by every element of $C$.  Then
$$\sum_{j = 0}^{d-1} \prod_{c \in C}
f_c(j) = 0.$$
\end{lemma}

\begin{proof}
We'll prove the statement by induction on the size of the set $C$.  Suppose for the base case that $C = \{c_1\}$.  Then there are $\frac{d}{c_1}$ values of $j$ divisible by $c$ and $d - \frac{d}{c_1}$ which aren't.  The expression in the lemma therefore equals
$$\frac{d}{c_1}(c_1-1) + \left(d - \frac{d}{c_1}\right)(-1) = d -
\frac{d}{c_1} - d + \frac{d}{c_1} = 0.$$

Now, suppose that the statement is true for sets of size $m$ and set $C = \{c_1,\ldots,c_m,c_{m+1}\}$.  Let $d$ be an integer divisible by each element of $C$.  Then we can break up the expression into two pieces based on whether $c_{m+1}$ divides $j$:
$$\sum_{j = 0}^{d-1} \prod_{i = 1}^{m+1}
f_{c_i}(j) = \sum_{\substack{1 \leq j < d, \\ c_{m+1}|j}} \prod_{i=1}^{m+1} f_{c_i}(j) + \sum_{\substack{1 \leq j < d, \\ c_{m+1} \nmid j}} \prod_{i=1}^{m+1} f_{c_i}(j).$$

Next, we may factor the $i = m+1$ term out of each product to obtain:
$$(c_{m+1}-1)\sum_{\substack{1 \leq j < d, \\ c_{m+1}|j}} \prod_{i=1}^{m} f_{c_i}(j) + (-1)\sum_{\substack{1 \leq j < d, \\ c_{m+1} \nmid j}} \prod_{i=1}^{m} f_{c_i}(j) = c_{m+1}\sum_{\substack{1 \leq j < d, \\ c_{m+1}|j}} \prod_{i=1}^{m} f_{c_i}(j) - \sum_{j=0}^{d-1} \prod_{i=1}^{m} f_{c_i}(j).$$

In the sum over multiples of $c_{m+1}$, we can replace each $j$ by $j/c_{m+1}$ without changing the value of the other $f_{c_i}$ because the $c_i$ are pairwise coprime.  Therefore, the last expression may be rewritten:
$$c_{m+1}\sum_{k=0}^{\frac{d}{c_{m+1}}-1} \prod_{i=1}^{m} f_{c_i}(k) - \sum_{j=0}^{d-1} \prod_{i=1}^{m} f_{c_i}(j).$$
Since both $d$ and $d/c_{m+1}$ are positive integers divisible by each element of the pairwise coprime set $\{c_1,\ldots,c_m\}$, the two sums above both equal zero by the inductive hypothesis.  This completes the proof.
\end{proof}

Using this lemma, we prove the properties of Theorem \ref{Hodgenumbers}. Let
$$S_{\ell} \coloneqq  \sum_{j = 0}^{d-1} \prod_{i,d|\ell a_i} f_{\frac{d}{a_i}}(j)$$
be the $\ell$th term of the sum \eqref{Hodgesum}, so that $H = \frac{1}{d}\sum_{j=0}^{d-1} S_{\ell}$. Notice that the product is indexed by a set depending only on $\ell$ (and not $j$).  The following lemma will show that $S_{\ell}$ is equal to zero for many values of $\ell$ in our examples.

\begin{lemma}
\label{elltermvanishing}
Let $X$ be a quasi-smooth Calabi-Yau hypersurface of degree $d$ with weights $a_0,\ldots,a_{n+1}$ all dividing $d$.  Suppose that the set $C \coloneqq \{d/a_i: a_i \neq 1\}$ is pairwise coprime.  Then in the notation above: if $\ell \neq 0$ and $c$ divides $\ell$ for some $c \in C$, then $S_{\ell} = 0$.
\end{lemma}

\begin{proof}
Let $c_i = d/a_i$ for every element $d/a_i \in C$.  Assuming that $\ell \neq 0$, we have $d \nmid \ell$ so the sum $S_{\ell}$ may be written
$$S_{\ell} = \sum_{j = 0}^{d-1} \prod_{i,d|\ell a_i} f_{\frac{d}{a_i}}(j) = \sum_{j=0}^{d-1} \prod_{c_i \in C, c_i | \ell} f_{c_i}(j).$$
Any subset of $C$ is pairwise coprime and all elements of $C$ divide $d$, so by Lemma \ref{counting}, $S_{\ell} = 0$ whenever the product is nonempty.  This occurs precisely when some $c_i$ divides $\ell$.
\end{proof}

Let $X$ satisfy the condition of Lemma \ref{elltermvanishing}. (For instance, $X = X_2^{(n)}$ or $X_1^{(n)}$; in these cases, $C = \{s_0,\ldots,s_n\}$ or $C = \{s_0,\ldots,s_{n-1}\}$, respectively.)  Then, $S_{\ell}$ is only nonzero if: a) $\ell = 0$, in which case the second product in $Q_{\ell}(x,y)$ is empty, or b) no $s_i$ divides $\ell$, in which case the first product in $Q_{\ell}(x,y)$ is empty.  In case a), $Q_0$ is a polynomial in $xy$.  In case b), $Q_{\ell}(x,y)$ is a monomial $x^a y^b$ with 
$$a + b = 2 \sum_{i = 0}^{n+1} \left(\frac{d}{2} - a_i \right) = 2\left(\frac{(n+2)d}{2} - d\right) = nd.$$

Therefore, the only orbifold Hodge numbers $h_{\mathrm{orb}}^{p,q}$ of the mirror of $X$ that could be nonzero satisfy $p = q$ or $p + q = n$.  Taking the mirror just switches these two types of Hodge numbers, so the vanishing in Theorem \ref{Hodgenumbers}(i) holds for both $X$ and its mirror.  This is enough to prove the first part of the theorem.  We also note that, in odd dimensions, all contributions to the middle orbifold cohomology of $X$ (i.e., to Hodge numbers $h^{p,p}$ of the mirror) come from $Q_0(x,y)$.

To prove parts (ii) and (iii) of Theorem \ref{Hodgenumbers}, we'll consider $X_2^{(n)}$ and $X_1^{(n)}$ separately: 

\textbf{Proof for $X_2^{(n)}$}: (hypersurface of degree $d = s_{n+1}-1$ in $\mathbb{P}(d/s_0,\ldots,d/s_n,1)$) 

The sum of all orbifold Betti numbers (equivalently,
the sum of all orbifold Hodge numbers) is $H = \frac{1}{d}(A + B)$, where $A \coloneqq S_0$ and $B$ is the sum of all $S_{\ell}$ such that $s_i \nmid \ell$, $i = 0,\ldots n$. First, we compute
\begin{align*}
A & = S_0 = \sum_{j = 0}^{d-1} f_{s_0}(j) \cdots f_{s_n}(j) f_d(j)\\
 & = (s_0 - 1) \cdots (s_n - 1)(d-1) + (-1)\sum_{j = 1}^{d-1}f_{s_0}(j) \cdots f_{s_n}(j).
\end{align*}
Here $S_0$ is divided into two pieces by explicitly evaluating the function $f_d$.  By Lemma \ref{counting},
$$\sum_{j = 1}^{d-1}f_{s_0}(j) \cdots f_{s_n}(j) = -f_{s_0}(0) \cdots f_{s_n}(0) = -(s_0 - 1) \cdots (s_n-1). $$
Therefore, $A$ simplifies to $d(s_0-1) \cdots (s_n-1).$  The number of values $\ell$ summed in $B$ is
$$d\left(1 - \frac{1}{s_0}\right) \cdots \left(1 - \frac{1}{s_n}\right) = (s_0 - 1) \cdots (s_n-1).$$

Since the product in $S_{\ell}$ is empty when no $s_i$ divides $\ell$, each $S_{\ell}$ term in $B$ equals $d$; the total is $B = d(s_0-1) \cdots (s_n-1)$. Therefore, $H = \frac{1}{d}(A+B) = 2(s_0-1)\cdots (s_n-1)$, which proves property (ii) for this example.  In odd dimensions, the middle orbifold cohomology has dimension $\frac{1}{d}A = (s_0-1) \cdots (s_n-1)$, which is half the total, as expected. Note that $X_2^{(n)}$ is self-mirror, so its orbifold Euler characteristic vanishes.

\textbf{Proof for $X_1^{(n)}$}: (hypersurface of degree $d = 2s_n - 2$ in $\mathbb{P}(d/s_0,\ldots,d/s_{n-1},1,1)$)

The idea of the proof is similar in this case. Once again, let $A \coloneqq S_0$ and $B$ be the sum of the $S_{\ell}$ with $s_i \nmid \ell$, $i = 0, \ldots n-1$.  The number of terms in $B$ is
$$d\left(1 - \frac{1}{s_0}\right) \cdots \left(1 - \frac{1}{s_{n-1}}\right) = 2(s_0 - 1) \cdots (s_{n-1}-1)$$
and each of these satisfies $S_{\ell} = d$.  So $B = 2d(s_0-1) \cdots (s_{n-1}-1)$.  Similarly,
\begin{align*}
A & = S_0 = \sum_{j = 0}^{d-1} f_{s_0}(j) \cdots f_{s_{n-1}}(j) (f_d(j))^2 \\
& = (s_0 - 1) \cdots (s_{n-1} - 1)(d-1)^2 + (-1)^2\sum_{j = 1}^{d-1}f_{s_0}(j) \cdots f_{s_{n-1}}(j).
\end{align*}
Applying Lemma \ref{counting} again, the last sum equals $-f_{s_0}(0) \cdots f_{s_{n-1}}(0) = (s_0 -1) \cdots (s_{n-1}-1)$.  Thus,
$$A = (s_0 - 1) \cdots (s_{n-1} - 1)d(d-2).$$
The total sum is
\begin{align*}
      H & = \frac{1}{d}(A + B) = \frac{1}{d}(d(d-2)(s_0 - 1) \cdots (s_{n-1} - 1) + 2d(s_0 - 1) \cdots (s_{n-1}-1)) \\
      & = (s_0-1)\cdots(s_{n-1}-1)d = 2(s_0-1) \cdots (s_n -1).
\end{align*}
The dimension of the middle orbifold cohomology in odd dimensions is $\frac{1}{d}A = (d-2)(s_0 - 1) \cdots (s_{n-1} - 1) = (s_0-1) \cdots (s_{n-1} - 1)(2s_n - 4)$.  The third example is the mirror of the first, so properties (i) and (ii) also hold for $X_3^{(n)}$.  The sums $A$ and $B$ switch under mirror symmetry, so the dimension of $H_{\mathrm{orb}}^n(X_3^{(n)},\Q)$ is $\frac{1}{d}B = 2(s_0 - 1) \cdots (s_{n-1} - 1)$.  The orbifold Euler characteristics of $X_1^{(n)}$ and $X_3^{(n)}$ are $\frac{1}{d}(B-A)$ and $\frac{1}{d}(A-B)$, respectively, using the vanishing result (i).
\end{proof}

\section{Klt Calabi-Yau surfaces with large index or small mld}

We have conjectured the klt Calabi-Yau pairs with standard coefficients
of largest index in each dimension,
and also the terminal or canonical Calabi-Yau varieties
of largest index. In section \ref{kltsection}, we conjecture
the answer to a natural intermediate problem: find
the klt Calabi-Yau {\it varieties} of largest index. In this section,
we prove that conjecture in dimension $2$,
finding the klt Calabi-Yau surfaces
with the largest index. We also find
the smallest minimal log discrepancy for klt Calabi-Yau surfaces.
The main tool is Brandhorst-Hofmann's classification of finite group
actions on K3 surfaces, extending many earlier results \cite{BH}.

\begin{proposition}
\label{dim2}
The largest index of any klt Calabi-Yau surface is $19$. The smallest
mld of any klt Calabi-Yau surface is $\frac{1}{13}$.
\end{proposition}

For comparison, the largest index of a klt Calabi-Yau
pair with standard coefficients of dimension $2$ is $66$
(Proposition \ref{dim2pair}), and the smallest mld
of a klt Calabi-Yau pair with standard coefficients of dimension $2$
is $\frac{1}{42}$ (Proposition \ref{42}, below).

\begin{proof}
For any klt Calabi-Yau surface $X$, let $Y$ be the index-1 cover of $X$
(section \ref{notation}). Thus $Y$ is a projective surface
with canonical singularities
such that $K_Y$ is trivial.
Writing $m$ for the index of $X$,
we have $X=Y/\mu_m$, where $\mu_m$ acts purely non-symplectically,
meaning that it acts faithfully on $H^0(Y,K_Y)\cong \mathbb{C}$. Also,
$\mu_m$ acts freely in codimension $1$ on $Y$, because $X$ is a Calabi-Yau
surface rather than a pair.

Let $\pi\colon Z\rightarrow Y$ be the minimal resolution. Since $K_Z=\pi^*(K_Y)$,
$Z$ is a smooth projective surface with $K_Z$ linearly equivalent to zero.
By Proposition \ref{equivterm},
the action of $\mu_m$ on $Y$ lifts to an action on $Z$.
By the classification of smooth projective
surfaces with trivial canonical bundle,
$Z$ is a $K3$ surface or an abelian surface.
If $Z$ is an abelian surface, then $m$ is at most $6$
\cite[Theorem C]{Blache}, \cite[Theorem 4.1]{DeQiZ}.
In that case, $X$ has mld at least $\frac{1}{6}$.

Since we are trying to find the largest index or smallest mld
among klt Calabi-Yau surfaces, we can assume
from now on that $Z$ is a smooth K3 surface, where more extreme values
can occur. Note that $Y$ may be a nontrivial
contraction of $Z$. Let $W=Z/\mu_m$, and let $(W,D)$ be the pair with standard coefficients
associated to the $\mu_m$-action on $Z$.
Denote by $\sigma\colon W \rightarrow X$ the 
quotient of the morphism $\pi\colon Z\to Y$
by the action of $\mu_m$, so that we have:
\begin{center}
    \begin{tikzcd}
    Z \arrow[r, "\pi"] \arrow[d] & Y \arrow[d] \\
    W \arrow[r, "\sigma"] & X.
    \end{tikzcd}
\end{center}

Kondo found that the group $\mu_{19}$ acts
purely non-symplectically on a certain smooth projective K3 surface $Z$
\cite[section 7]{Kondo}.
In fact, the K3 surface and the action are unique up to isomorphism
and automorphisms of $\mu_{19}$
\cite[Theorem 1.1]{Brandhorst}. Also, $\mu_{19}$ acts freely
in codimension 1 on $Z$ \cite[Table 37]{BH}.
So $W = Z/\mu_{19}$ is a klt Calabi-Yau surface of index $19$.
In the spirit of this paper,
we can describe $Z$ as the minimal resolution of the hypersurface
$Y_{10}\subset \mathbb{P}^3(5,3,1,1)$ defined by
$0=x_0^2+x_1^3x_3+x_3^9x_2+x_2^7x_1$. Here $\mu_{19}$ acts on $Y$
by
$$\zeta[x_0:x_1:x_2:x_3]=[\zeta x_0:\zeta^7 x_1:\zeta^2 x_2:x_3],$$
$Y$ has an $A_2$ singularity,
and $Y/\mu_{19}$ is also a klt Calabi-Yau surface of index $19$.

In the other direction,
Appendix B in \cite{BH} shows that for each purely non-symplectic action
of a cyclic group of order greater than $19$ on a smooth K3 surface,
there is a smooth curve of positive genus that is fixed by a nontrivial subgroup. This curve cannot be contracted by $\pi\colon Z\to Y$ (in the notation above), and so $\mu_m$ cannot act freely in codimension $1$ on $Y$. This contradicts that $X = Y/\mu_m$ is a klt
Calabi-Yau surface (rather than a Calabi-Yau pair).
So $19$ is the largest index among all klt Calabi-Yau surfaces.

Next, we will find the smallest possible mld for a klt Calabi-Yau surface $X$.
We know that $X$ has index $m \leq 19$. As above, let $Y$ be the index-1
cover of $X$, $Z$ its minimal resolution, and $(W,D)$ the quotient
of $Z$ by $\mu_m$.

\begin{lemma}
\label{samemld}
$$\mld(X)=\mld(W,D).$$
\end{lemma}

\begin{proof}
Consider the commutative square above.
Since $Y$ has canonical singularities, we have $K_Z=\pi^*K_Y$.
Taking quotients by $\mu_m$, we have that
$K_Y$ is the pullback of $K_X$ and $K_Z$ is the pullback
of $K_W+D$. So $K_W+D$ and $\sigma^*K_X$ both pull back
to $K_Z$ on $Z$. Since they differ by an explicit $\Q$-divisor
on $W$, it follows that $K_W+D=\sigma^*K_X$.
Since the mld of $X$ can be computed on a log
resolution of $(W,D)$, we conclude that $\mld(X)=\mld(W,D)$.
\end{proof}

Given a pair $(X,D)$, we'll use the notation $\mld_{\geq 2}(X,D)$ to denote the minimal log discrepancy over all points
$x \in X$ with codimension at least $2$ (corresponding to exceptional divisors over $X$, rather than all divisors over $X$). We have 
the following general bound on discrepancies for quotients by finite groups.

\begin{proposition}
\label{discrepformula}
Let $(Z,F)$ be a quasi-projective terminal pair with a faithful action of a finite group $G$. Let $W=Z/G$, and let $D$ be the $\Q$-divisor on $W$
such that $K_Z+F$ is the pullback of $K_W+D$. Then
$$\mld_{\geq 2}(W,D) > \frac{1}{|G|}.$$
If $Z = (Z,0)$ is a terminal Gorenstein variety, then
$$ \mld_{\geq 2}(W,D) \geq \frac{2}{|G|}.$$
\end{proposition}

We need to exclude points of codimension 1 in these inequalities. For example,
if a cyclic group $G=\mu_m$
fixes an irreducible divisor $S$ in a variety $Z$, then the image $D_1$ of $S$
in $W$ occurs in $D$ with coefficient $(m-1)/m$, and so
the generic point $w$ of $D_1$ has $\mld_w(W,D)=1/m=1/|G|$.

\begin{proof}
By \cite[Corollary 2.43]{Kollarsing}, we have the inequalities
$$\mld_{\geq 2}(W,D) \leq \mld_{\geq 2}(Z,F) \leq |G|\mld_{\geq 2}(W,D).$$
(The notation ``$\mathrm{discrep}(W,D)$" of \cite{Kollarsing}
means the infimum of all discrepancies
of exceptional divisors over $X$,
which is the same as $\mld_{\geq 2}(W,D) - 1$.) By definition, the terminality of $(Z,F)$ means that
$\mld_{\geq 2}(Z,F)>1$. It follows that
$$\mld_{\geq 2}(W,D) \geq \frac{1}{|G|} \mld_{\geq 2}(Z,F) > \frac{1}{|G|}.$$

If $Z=(Z,0)$ is a terminal Gorenstein variety, then $|G|(K_{W} + D)$ is Cartier.  It follows that the discrepancy of every
divisor over $W$ is a multiple of $1/|G|$, and so $\mld_{\geq 2}(W,D) \geq 2/|G|$.
\end{proof}

By \cite[Table 37]{BH}, there is a smooth K3 surface $Z$ with a purely non-symplectic action of the cyclic group $\mu_{13}$; the fixed locus
consists of a smooth rational curve $C$ and 9 other points. Since $K_Z$
is trivial, the adjunction formula gives that $C$ is a $(-2)$-curve.
Let $Y$ be the canonical K3 surface obtained by contracting $C$.
Then the cyclic group $\mu_{13}$ acts
freely in codimension 1 on $Y$. So $X \coloneqq Y/\mu_{13}$ is a klt Calabi-Yau surface of index $13$.

For an explicit example,
start with the canonical K3 surface $S_{11}\subset
\mathbb{P}^3(5,3,2,1)$ defined by $0=x_0^2x_3+x_1^3x_2+x_2^3x_0
+x_3^8x_1$, on which there is a purely non-symplectic
action of $\mu_{13}$:
$$\zeta[x_0:x_1:x_2:x_3]=[\zeta x_0:\zeta^2 x_1:\zeta^{-4}x_2:x_3].$$
This action is free in codimension 1, and so $S/\mu_{13}$ is a klt
Calabi-Yau surface of index 13. One checks that
the minimal resolution $Z$ of $S$
contains a smooth rational curve $C$ fixed by $\mu_{13}$
over the $A_4$ singularity $[1:0:0:0]$ of $S$,
in agreement with \cite[Table 37]{BH}.

\begin{lemma}\label{13}
There is a klt Calabi-Yau surface of index $13$.
Every such surface has mld $\frac{1}{13}$.
\end{lemma}

\begin{proof}
We have shown that there is a klt Calabi-Yau surface $X$ of index 13.
As throughout this section,
let $Y$ be the index-1 cover of $X$, $Z$ the minimal resolution
of $Y$, and $(W,D)=Z/\mu_{13}$.
By Lemma \ref{samemld}, $\mld(X)=\mld(W,D)$.
Here $13(K_{W}+D)$
is Cartier, and so $\mld(W,D)$ is a positive integer divided by $13$.
Also, $D$ is $\frac{12}{13}$ times the image of the curve $C$ above,
and so $(W,D)$ has log discrepancy $\frac{1}{13}$ at the generic point
of this curve. Thus $\mld(W,D)=\frac{1}{13}$.
\end{proof}

\begin{lemma}
\label{14}
The mld of a klt Calabi-Yau surface of index $14$ is $\frac{1}{7}$.
\end{lemma}

\begin{remark}
Although it is not needed for the proof of Proposition \ref{dim2},
we remark that there is a klt Calabi-Yau surface of index $14$.
Start with the canonical K3 surface $S_7\subset \mathbb{P}^3(3,2,1,1)$
defined by $0=x_0^2x_3+x_1^3x_2+x_1x_2^5+x_3^7$, on which
there is a purely non-symplectic action of $\mu_{14}$:
$$\zeta[x_0:x_1:x_2:x_3]=[\zeta x_0:\zeta x_1:\zeta^{-3}x_2:\zeta^{-2}x_3].$$
The curve $C=\{0=x_1=x_3\}\cong \mathbb{P}^1(3,1)$ is contained in $S$
and fixed by the subgroup $\mu_2$. Here $C^2=-4/3<0$, and so $C$
can be contracted, by Artin's contraction theorem applied to the minimal
resolution of $S$ \cite[Theorem 2.7]{Artin}.
This yields another canonical K3 surface $Y$
on which $\mu_{14}$ acts freely in codimension $1$. Then $X=Y/\mu_{14}$
is a klt Calabi-Yau surface of index $14$.
\end{remark}

\begin{proof}
(Lemma \ref{14})
Let $X$ be a klt Calabi-Yau surface of index $14$. In the notation above,
$Z$ must be a smooth K3 surface with a purely non-symplectic action
of $\mu_{14}$, and every curve in $Z$ fixed by a nontrivial subgroup
of $\mu_{14}$ must have genus zero. By \cite[Table 15, 0.14.0.4]{BH},
there is exactly one deformation type of smooth K3 surfaces $Z$
with such an action, up to automorphisms of $\mu_{14}$.
The table shows that there is no curve in $Z$
fixed by the whole group $\mu_{14}$, but there
are three smooth curves of genus zero fixed by the subgroup $\mu_2$
and one smooth curve of genus zero fixed by $\mu_7$.

Let $(W,D)=Z/\mu_{14}$.
We have $\mld(X)=\mld(W,D)$ by Lemma \ref{samemld}. Here $D$ has coefficients
$6/7$ and $1/2$. So the mld of $(W,D)$ is $1/7$
at one point of codimension 1 and greater than $1/7$ at every other point
of codimension 1. Also, by Proposition \ref{discrepformula}, since $(W,D)$ is the 
quotient of a smooth K3 surface by $\mu_{14}$, the mld of $(W,D)$ at any closed point 
is at least $2/14=1/7$. So the mld of $X$ is $1/7$.
\end{proof}

\begin{lemma}
\label{1719}
For $m$ equal to $17$ or $19$,
the mld of a klt Calabi-Yau surface of index $m$ is at least $\frac{2}{m}$.
\end{lemma}

\begin{proof}
Let $X$ be a klt Calabi-Yau surface of index $m$ equal to $17$
or $19$. Let $Y$ be the index-1 cover of $X$, $Z$
the minimal resolution of $Y$, and $(W,D)=Z/\mu_m$.
By \cite[Appendix B]{BH}, for $m$ equal to $17$ or $19$,
every purely non-symplectic action of $\mu_m$ on a smooth
K3 surface is free in codimension $1$; so $D=0$.
Then Lemma \ref{samemld} gives that $\mld(X)=\mld(W)$. Since $W$
is a normal variety (rather than a pair),
$\mld(W)$ is equal to the minimum of its mld's
at closed points, and these are at least $\frac{2}{m}$ by Proposition
\ref{discrepformula}.
\end{proof}

We now show that among all klt Calabi-Yau surfaces $X$,
the smallest mld is $1/13$. We know by Lemma \ref{13}
that $1/13$ does occur. Here $X$ must be a quotient
$Y/\mu_m$ as discussed above, in particular with $\mu_m$ acting
freely in codimension $1$. The mld of $X$ is a positive integer
divided by $m$. So if $X$ has mld less than $1/13$,
it must have index $m\geq 14$.
Since the minimal resolution $Z\to Y$ contracts only smooth
rational curves, every curve in $Z$ fixed by a nontrivial
subgroup of $\mu_m$ must have genus zero.
By \cite[Appendix B]{BH}, if a smooth K3 surface $Z$
has a purely non-symplectic action of the cyclic group $\mu_m$ with $m$
at least 14 and not equal to $14$, $17$, or $19$,
then there is a curve of genus $>0$ fixed
by a nontrivial subgroup. So $m$ must be $14$, $17$, or $19$.
Therefore, Lemmas \ref{14} and \ref{1719} give
that the smallest mld of a klt Calabi-Yau surface is $1/13$.
Proposition \ref{dim2} is proved.
\end{proof}

We conclude with the analogous bound for pairs.

\begin{proposition}
\label{42}
The smallest mld of any klt Calabi-Yau pair with standard
coefficients of dimension 2 is $\frac{1}{42}$.
\end{proposition}

\begin{proof}
By Theorem \ref{smallmld}, there is a klt Calabi-Yau pair
with standard coefficients of dimension 2 that has mld $1/42$.

Suppose that there is a klt Calabi-Yau pair $(X,E)$ with standard coefficients
of dimension 2 that has mld less than $1/42$. Let $m$ be the index
of $(X,E)$. The mld is a positive integer divided by $m$, and so we must
have $m>42$. By Proposition \ref{dim2pair}, we have $m\leq 66$.
Let $Y$ be the index-1 cover of $(X,E)$, so that $(X,E)=Y/\mu_m$
with $\mu_m$ acting purely non-symplectically. Let $Z$ be the minimal
resolution of $Y$, and let $(W,D)=Z/\mu_m$. Here $Z$ is a K3 surface.
The proof of Lemma \ref{samemld} shows that $\mld(X,E)=\mld(W,D)$.

By Brandhorst and Hofmann's classification
of purely non-symplectic
cyclic group actions on K3 surfaces, $m$ must be 44, 48, 50, 54, or 66
\cite[Appendix B]{BH}. In each case, the classification shows that
the subspace of $Z$ fixed by $G$ has dimension zero. Equivalently,
each curve in $Z$ is fixed by a proper subgroup of $G$.
So the coefficients of $D$ have the form $(b-1)/b$ with
$b\leq m/2$. Equivalently,
the mld of $(W,D)$ at each point of codimension 1 in $X$
is at least $2/m$. By Proposition \ref{discrepformula},
the mld of $(W,D)$ at each closed point is also at least $2/m$. So
$\mld(W,D)$ is at least $2/m$, hence at least $1/33$, contradicting
that it is less than $1/42$.
\end{proof}

\section{Higher-dimensional klt Calabi-Yau varieties with large index}
\label{kltsection}

In this section, we find klt Calabi-Yau varieties in every dimension at least $2$,
which we conjecture have the largest possible index.  In dimension $2$, the
construction reproduces the index $19$ example mentioned in the proof of Proposition $\ref{dim2}$.
Each of our varieties is constructed as a quotient of a weighted projective
hypersurface by a cyclic group action.  To define these actions, we'll use 
the techniques of Artebani-Boissi\`ere-Sarti \cite{ABS} to write down
hypersurfaces with an action of a large cyclic group. We only manage
to compute the index of our example exactly in dimensions at most 30.
So, at least in low dimensions, we construct klt Calabi-Yau varieties
of extremely large index.

First, we prove some general properties about hypersurfaces defined by loop potentials,
some of which appear in \cite{ABS}.
A {\it potential }is a sum of $n$ monomials in $n$ variables.
A {\it loop potential} has the form
\begin{equation}
\label{loop_potential}
   W \coloneqq x_1^{b_1}x_2 + \cdots + x_{n-1}^{b_{n-1}}x_n + x_n^{b_n}x_1 
\end{equation}
for some integers $b_i \geq 2$.  Suppose that for some positive
integer weights $a_1,\ldots,a_n$ on the variables, $W$
is homogeneous of degree $d:=\sum_i a_i$. Then the hypersurface $X \coloneqq \{W = 0\}$ in $\mathbb{P}(a_1,\ldots,a_n)$ is quasi-smooth and Calabi-Yau.  Writing $W$ in the form 
$$W = \sum_{i=1}^n \prod_{i=1}^n x_i^{a_{ij}},$$
define an associated matrix $A = (a_{ij})$:
$$A = \begin{pmatrix}
b_1 & 1 & & &\\
& b_2 & 1 & &\\
& & \ddots & \ddots \\
 & & & b_{n-1} & 1 \\
1 & & & & b_n
\end{pmatrix}
.$$
Let $\mathrm{Aut}(W)$ be the group of diagonal automorphisms of $\C^n$ preserving the potential $W$.  By the results of \cite{ABS}, we may identify $\mathrm{Aut}(W)$ with $A^{-1} \Z^n / \Z^n.$ That is, the columns of $A^{-1}$ generate $\mathrm{Aut}(W)$ modulo $\Z^n$, where an element $(c_1,\ldots,c_n) \in \mathrm{Aut}(W)$ acts on $\C^n$ by multiplication by the diagonal matrix $\mathrm{diag}(e^{2 \pi i c_1},\ldots, e^{2 \pi i c_n}).$

This action descends to an action on the weighted projective hypersurface $X$, giving a surjective homomorphism $\alpha\colon \mathrm{Aut}(W) \rightarrow \mathrm{Aut}_T(X)$ to the group of toric automorphisms of $X$.

\begin{lemma}[{cf. \cite[Proposition 2]{ABS}}]
\label{loopmatrix}
The group $\mathrm{Aut}(W)$ has order $\Gamma  \coloneqq  \det(A) = (-1)^{n+1} + b_1 \cdots b_n$.  The last column of $A^{-1}$ is
$$v_n = \frac{1}{\Gamma}\begin{pmatrix}
(-1)^{n-1} \\
(-1)^{n-2} b_1 \\
(-1)^{n-3} b_1 b_2 \\
\vdots \\
b_1 \cdots b_{n-1}
\end{pmatrix}.$$
The other columns $v_i$ of $A^{-1}$ satisfy the following relations modulo $\Z^n$: $v_i = (-1)^{n-i} b_n b_{n-1} \cdots b_{i+1} v_n$, $i = 1,\ldots,n-1$.  In particular, $\mathrm{Aut}(W)$ is cyclic with generator $\varphi  \coloneqq  v_n$ (or any one of the other columns).  The first row of $A^{-1}$ is 
$$w_1 = \frac{1}{\Gamma}\begin{pmatrix}
b_2 \cdots b_n &
-b_3 \cdots b_n &
b_4 \cdots b_n &
\cdots &
(-1)^{n-1}
\end{pmatrix}.$$
The other rows $w_i$ of $A^{-1}$ satisfy the following relations modulo $\Z^n$: $w_i = (-1)^{i-1} b_1 b_2 \cdots b_{i-1}$, $i = 2,\ldots n$.
\end{lemma}

The kernel of $\alpha \colon \mathrm{Aut}(W) \rightarrow \mathrm{Aut}_T(W)$ is a subgroup $J_W \subset \mathrm{Aut}(W)$ of order $d$ generated by the vector of charges $(q_1,\ldots,q_n)$, where each {\it charge} $q_i$ is the sum of the entries of the $i$th row of $A^{-1}$.  The charges have the property $q_i = a_i/d$.  Lemma \ref{loopmatrix} implies that 
$$q_1 = \frac{1}{\Gamma}\left((-1)^{n-1} + (-1)^{n-2} b_n + \cdots + b_2 \cdots b_n\right).$$
Modulo $\Z^n$, the other charges are multiples of this one, so that the degree $d$ equals the denominator of $q_1$:
\begin{equation}
\label{degree-of-loop}
    d = \frac{\Gamma}{\gcd(\Gamma, (-1)^{n-1} + (-1)^{n-2} b_n + \cdots + b_2 \cdots b_n)}.
\end{equation}
The same reasoning shows that the degree $d^{\mathsf{T}}$ of the mirror hypersurface (which is defined by the potential $W^{\mathsf{T}}$ associated to the matrix $A^{\mathsf{T}}$) is
\begin{equation}
\label{degree-of-transpose-loop}
d^{\mathsf{T}} = \frac{\Gamma}{\gcd(\Gamma, (-1)^{n-1} + (-1)^{n-2} b_1 + \cdots + b_1 \cdots b_{n-1})}.
\end{equation}
Loop potentials are useful for constructing cyclic group actions which are free in codimension 1, in light of the following proposition.

\begin{proposition}
\label{freecodim1}
If a hypersurface $X$ with $\dim(X) \geq 3$ is defined by a loop potential, the action of $\mathrm{Aut}_T(X)$ is free in codimension $1$.
\end{proposition}

\begin{proof}
First, we reduce to only checking the stabilizers of coordinate hyperplanes.  Indeed, $\mathrm{Aut}_T(X)$ is a subgroup of the torus action on $\mathbb{P}(a_1,\ldots,a_n)$, so the only nontrivial stabilizers of the action of $\mathrm{Aut}_T(X)$ on weighted projective space occur on coordinate strata in the complement of the open torus.  Further, we claim that the intersection of $X$ with a toric stratum $S$ is only codimension $1$ in $X$ if $S$ is codimension $1$ in $\mathbb{P}(a_1,\ldots,a_n)$.  Indeed, if $X$ contained $S$ of codimension $2$ in weighted projective space, then the potential $W = x_1^{b_1}x_2 + \cdots + x_{n-1}^{b_{n-1}}x_n + x_n^{b_n}x_1$ would become identically zero after setting some pair of variables $x_i$ and $x_j$ equal to zero.  Since $\dim(X) \geq 3$, we have $n \geq 5$, and no choice of $i$ and $j$ makes the loop potential zero.

Thus, it's enough to compute the stabilizers of each $H_i  \coloneqq  X \cap \{x_i = 0\}$.  Without loss of generality we'll just show the stabilizer of $H_n$ is trivial.  To this end, let $H \subset \mathrm{Aut}_T(W)$ be the subgroup fixing $H_n$ and let $G = \alpha^{-1}(H)$.  Then $G$ is a cyclic subgroup of $\mathrm{Aut}(W)$.  We want to show that in fact $G = J_W$.

For an element $g \coloneqq (c_1,\ldots,c_n)$ in $\mathrm{Aut}(W)$, the statement that the element $\alpha(g)$ fixes $H_n$ is equivalent to the following statement in terms of coordinates on $\C^n$: for all $x_1,\ldots,x_{n-1} \in \C$, there exists a $t \in \C^*$ such that 
$$g \cdot (x_1,\ldots,x_{n-1},0) = t \cdot (x_1,\ldots,x_{n-1},0),$$
where the $\C^*$ action is $t \cdot (y_1,\ldots,y_n) = (t^{a_1} y_1,\ldots,t^{a_n} y_n)$.  Now, the action of $\mathrm{Aut}(W)$ is by diagonal matrices whose entries are unit complex numbers, so it's immediate that $|t| = 1$ and in fact that $t = e^{i \theta}$ for $\theta$ a rational multiple of $\pi$.  In order to compute in terms of elements of $(\Q/\Z)^n$, it will be convenient to express the element $t \in \C^*$ as $r (q_1,\ldots,q_n) \in (\Q/\Z)^n$, where $(q_1,\ldots,q_n)$ is the vector of charges and $r$ is some rational number. (Remember that we use the convention that each vector $(c_1,\ldots,c_n) \in (\Q/\Z)^n$ corresponds to the action by the diagonal matrix $\mathrm{diag}(e^{2 \pi i c_1},\ldots, e^{2 \pi i c_n})$).

There is a point of $H_n$ with all coordinates $x_1,\ldots,x_{n-1}$ nonzero.  Otherwise, since $H_n$ is codimension $1$ in $X$, it must contain some entire stratum $\{x_i = x_j = 0\}$, which we saw above cannot occur when $\dim(X) \geq 3$.  Thus, if the element $g$ fixes $H_n$, there is an $r \in \Q$ such that the following congruences hold:
\begin{equation}
\label{fixedhyperplane}
r q_i = c_i \bmod \Z, 1 \leq i \leq n-1.
\end{equation}
We already know that the vector of charges itself generates $J_W$, so if $r \in \Z$, $r(q_1,\ldots,q_n)$ will equal an element of $J_W$.  We will derive a contradiction from the following assumption: there exists $g \in \mathrm{Aut}(W)$ and $r$ {\it not} an integer such that the congruences \eqref{fixedhyperplane} hold.  This will show that $G = J_W$, as required.

First, reduce to the case that $r = 1/A$ for a positive integer $A \geq 2$.  Indeed, we can always replace $r$ and $g$ by the same integer multiple of each and the congruences \eqref{fixedhyperplane} will still hold; we can also add an integer to $r$ and multiply $g$ by the corresponding element of $J_W$.  Using such adjustments, any rational number $r = B/A$ in lowest terms can be turned into $1/A$.  Consider the first two coordinates $c_1$ and $c_2$ of $g$.  We know that $c_2 \equiv -b_1 c_1 \bmod \Z$, since this relation holds for the generator $\varphi$ of the group $\mathrm{Aut}(W)$ and $g$ is some multiple of $\varphi$.  But $c_i \equiv rq_i \bmod \Z$, so we have $rq_2 \equiv -b_1 rq_1 \bmod \Z$.  Therefore, $s \coloneqq b_1 r q_1 + r q_2$ is an integer.  Substituting $r = 1/A$ and using the definition of charges $q_i = a_i/d$, we have 
$$b_1 a_1 + a_2= A d s.$$
However, our loop potential is homogeneous of degree $d$ and contains the monomial $x_1^{b_1} x_2$ so $b_1 a_1 + a_2 = d$.  This would mean that $A s = 1$, contradicting the assumption that $A \geq 2$ and $s$ is an integer.  This completes the proof.
\end{proof}

Next, we'll state a criterion for whether the induced action of $\mathrm{Aut}_T(X)$ on $H^0(X,K_X) \cong \C$ is faithful. This in turn implies that the quotient $X/\mathrm{Aut}_T(X)$ has index $|\mathrm{Aut}_T(X)| = \Gamma/d$.  Let $\mathrm{SL}(W) \subset \mathrm{Aut}(W)$ be the subgroup with determinant $1$.  Since $X$ is Calabi-Yau, we have $J_W \subset \mathrm{SL}(W)$.  The image $\alpha(\mathrm{SL}(W)) \subset \mathrm{Aut}_T(X)$ is the kernel of the action on $H^0(X,K_X)$.  By \cite[Corollary 1]{ABS}, $|\mathrm{SL}(W)| = \Gamma/d^{\mathsf{T}}$.  Therefore, we have the following criterion:

\begin{proposition}
\label{faithful}
The group $\mathrm{Aut}_T(X)$ acts faithfully on $H^0(X,K_X)$ if and only if $d \cdot d^{\mathsf{T}} = \Gamma$.
\end{proposition}

\begin{proof}
The kernel of $\alpha$ is $J_W$, so the action is faithful if and only if $\mathrm{SL}(W) = J_W$.  Since $J_W \subset \mathrm{SL}(W)$, this holds exactly when the two groups have the same order.  But $|J_W| = d$, so this is equivalent to $d = \Gamma/d^{\mathsf{T}}$.
\end{proof}

Nearly identical arguments apply to any $X$ with $\dim(X) \geq 3$ defined by a polynomial of the form $W = x_0^{b_0} + f_{\mathrm{loop}}$, where $f_{\mathrm{loop}}$ is of the form of \eqref{loop_potential}.  The corresponding matrix $A$ is block diagonal, so there is a subgroup $G \cong \Z/\Gamma \Z$ of $\mathrm{Aut}(X)$ with index $b_0$ whose image in $\mathrm{Aut}_T(X)$ acts freely in codimension $1$.  The quotient of $X$ by this cyclic group is therefore a klt Calabi-Yau variety (rather than a pair).  We can check whether this action is purely non-symplectic using a criterion very similar to Proposition \ref{faithful}.

Next, we define several sequences of numbers based on Sylvester's sequence for each dimension $n$. These will be used to define the exponents and weights for our example of large index. 

\begin{definition}\label{b,r+i}
For $n=2r+1$, $r\geq 1$ or $n=2r$, $r\geq 1$, define integers
$b_0,\ldots,b_{n-1}$ as follows.  Set $b_i \coloneqq s_i$ when
$0 \leq i \leq r$, and define $b_{r+i}$ by the inductive formula:
\begin{multline*}
b_{r+i} \coloneqq 1 + (b_{r+1-i}-1)^2[1 + (b_{r}-1)b_{r+1}\\
+(b_{r-1}-1)b_{r}b_{r+1}b_{r+2}
+\cdots\
+(b_{r+2-i}-1)b_{r+3-i} \cdots b_{r+i-1}]
\end{multline*}
for $1\leq i\leq r$ when $n=2r+1$ or $1\leq i\leq r-1$ when $n=2r$.
\end{definition}

\begin{definition}
For $n=2r+1$, $r\geq 1$, define an integer $d$ by:
\begin{equation*}
\begin{split}
 d & \coloneqq (b_0-1)^2b_{1} \cdots b_{2r}+ (b_{1}-1)[1 + (b_{r}-1)b_{r+1}  \\
&+(b_{r-1}-1)b_rb_{r+1}b_{r+2}+\cdots+(b_{1}-1)b_2\cdots b_{2r}].
\end{split}
\end{equation*}
For $n=2r$, $r\geq 1$, define an integer $d$ by:
\begin{equation*}
\begin{split}
d& \coloneqq (b_1-1)^2b_{2}\cdots b_{2r-1}+(b_2-1)[1+(b_r-1)b_{r+1}\\
&+(b_{r-1}-1)b_rb_{r+1}b_{r+2}+\cdots+(b_2-1)b_{3}\cdots b_{2r-1}].
\end{split}
\end{equation*}
\end{definition}

\begin{definition}
For $n=2r+1$, $r\geq 1$, define $b_{2r+1} \coloneqq \frac{d+1}{2}$.

For $n=2r$, $r\geq 1$, define $b_{2r} \coloneqq \frac{2d+1}{3}$.
\end{definition}

\begin{definition}\label{ai}
For $n=2r+1$, $r\geq 1$, define $a_j$ for $0\leq j \leq2r+2$ by the following
inductive formulas. They are positive integers.
\begin{equation*}
\begin{split}
&a_{2r+1}=a_{2r+2} \coloneqq 1, \\
&a_{2r-i} \coloneqq a_{2r-i+1}\left(\frac{1}{s_{i+1}}\right)+(s_{r+1}-1) \left(1-\frac{1}{s_{i+1}}\right) \bigg[ a_{2r-i+1}\left(1-\frac{1}{s_{i+1}}\right)+\cdots\\
&+ a_{2r}\left(1-\frac{1}{s_2}\right)+ a_{2r+1}\left(1-\frac{1}{s_1}\right) + a_{2r+2}\left(1-\frac{1}{s_0} \right) \bigg] b_{r+1}\cdots b_{2r-i-1} \text{ for } 0 \leq i\leq r-1,\\
&a_i \coloneqq d-a_{2r+1-i}b_{2r+1-i} \text{ for } 0\leq i \leq r.
\end{split}
\end{equation*}

For $n=2r$, $r\geq 1$, define $a_j$ for $0\leq j \leq 2r+1$ by the inductive formulas:
\begin{equation*}
\begin{split}
&a_{2r} = a_{2r+1} \coloneqq 1\\
&a_{2r-1-i} \coloneqq a_{2r-i}\left(\frac{1}{s_{i+2}}\right)+(s_{r+1}-1) \left(1-\frac{1}{s_{i+2}}\right) \bigg[ a_{2r-i}\left(1-\frac{1}{s_{i+2}}\right)+\cdots \\
&+a_{2r-1} \left(1-\frac{1}{s_3}\right)+a_{2r}\left(1-\frac{1}{s_2}\right) + a_{2r+1}\left(1-\frac{1}{s_1}\right) \bigg] b_{r+1}\cdots b_{2r-i-2} \text{ for } 1\leq i \leq r-2,\\
&a_{i} \coloneqq d-b_{2r+1-i}a_{2r+1-i} \text{ for } 0\leq i \leq r,\\
&a_0 \coloneqq \frac{d}{2}.
\end{split}
\end{equation*}
\end{definition}

\begin{definition}
For $n=2r+1$, $r\geq 1$, we define an integer $b_{2r+2} \coloneqq d-a_{r+1}$.

For $n=2r$, $r\geq 1$, we define an integer $b_{2r+1} \coloneqq d-a_{r+1}$.
\end{definition}

\begin{definition}
For $n=2r+1$, $r\geq 1$, define an integer $m$ by:
\begin{equation*}
\begin{split}
m& \coloneqq b_0\cdots b_{2r+1}-b_0\cdots b_{r}b_{r+2}\cdots b_{2r+1}\\
&+ b_0\cdots b_{r-1}b_{r+2}\cdots b_{2r+1} - b_0\cdots b_{r-1}b_{r+3}\cdots b_{2r+1}+ \cdots - b_0 + 1.
\end{split}
\end{equation*}
For $n=2r$, $r\geq 1$, define an integer $m$ by:
\begin{equation*}
\begin{split}
m & \coloneqq b_1\cdots b_{2r} - b_1\cdots b_rb_{r+2}\cdots b_{2r}+ b_1\cdots b_{r-1}b_{r+2}\cdots b_{2r} - \cdots + 1.
\end{split}
\end{equation*}
\end{definition}

We were led to the choice of weights $a_i$ and exponents $b_i$ in Definition \ref{ai} and Definition \ref{b,r+i} by an inductive procedure. First we chose
the general form of the equation in \eqref{oddeq} or \eqref{eveneq} below. Since each monomial in the equation must have degree $d$, we get a relation involving $d$ and some $a_i$ and $b_i$. We use the relations given by all monomials and the relation that all the weights $a_i$ sum to $d$ to express $d$ as $r+1$ linear combinations of $a_{r+i},\ldots, a_{2r+2}$ for $1\leq i \leq r+1$ and express each $a_{r+i}$ as a linear combination of $a_{r+i+1},\ldots,a_{2r+2}$ for $1\leq i \leq r$. During each step, we choose one weight as big as possible which determines our choice of the corresponding exponent. In each step, we can increase $i$ by one. This process yields the formulas above.

We now define our example. When $n=2r+1$ for $r\geq 1$,
let $X_d$ be the hypersurface in $\mathbb{P}(a_0,a_1,\ldots, a_{2r+1}, a_{2r+2})$ of degree $d$ given by
\begin{equation}\label{oddeq}
x_0^{b_0}x_{2r+2}+ x_1^{b_1} x_{2r+1} + \cdots + x_r^{b_r}x_{r+2}+x_{r+1}^{b_{r+1}}x_r + x_{r+2}^{b_{r+2}} x_{r-1} + \cdots +x_{2r+1}^{b_{2r+1}}x_0
       +x_{2r+2}^{b_{2r+2}}x_{r+1} = 0.
\end{equation}

This is a loop potential, but with the variables written in a different order
than in \eqref{loop_potential} above; one can check that it is homogeneous of weighted degree $d$.
The group $G = \mathrm{Aut}_T(X)$ is cyclic of order $m$ and acts freely in codimension $1$ on $X$ by Proposition \ref{freecodim1}.
Similarly, when $n = 2r$ for $r \geq 1$,
let $X_d$ be the hypersurface in $\mathbb{P}(a_0,a_1,\ldots, a_{2r}, a_{2r+1})$ of degree $d$ given by
\begin{equation}\label{eveneq}
x_0^{b_0} + x_1^{b_1}x_{2r+1}+ x_2^{b_2} x_{2r} + \cdots + x_r^{b_r}x_{r+2}+x_{r+1}^{b_{r+1}}x_r
       + x_{r+2}^{b_{r+2}} x_{r-1} + \cdots +x_{2r}^{b_{2r}}x_1+x_{2r+1}^{b_{2r+1}}x_{r+1} = 0.
\end{equation}

This potential has the form $x_0^2 + f_{\mathrm{loop}}$.  There is a subgroup $G$ of index $2$ inside $\mathrm{Aut}_T(X)$ corresponding to $f_{\mathrm{loop}}$ which is cyclic of order $m$ and which acts freely in codimension $1$ on $X$ (this does not follow 
from Proposition \ref{freecodim1} when $n = 2$, but nevertheless is still true in this case).

We conjecture that the subgroup $G$ acts faithfully on
$H^0(X,K_X)\cong \C$,
so that the index of $X/G$ is in fact $m$. By Proposition \ref{faithful}, proving this conjecture would amount to
a gcd calculation involving the exponents $b_i$. This is true
in dimensions at most $30$ by computer verification, and so we have
klt Calabi-Yau varieties of extremely large index at least
in low dimensions. The number $m$
has $\log m$ asymptotic to $\log s_{n+1}$, and so this conjecture
would imply that the index of our klt Calabi-Yau $n$-fold
is comparable to that of the klt Calabi-Yau pair
in Theorem \ref{klt-index},
and much bigger than the index of our terminal Calabi-Yau
variety in Corollary \ref{can-term-ex}. More precisely, we conjecture:

\begin{conjecture}
\label{kltCYconj}
The quotient $X/G$ above has index $m$.
Moreover, in each dimension $n\geq 2$,
this is the largest possible index for a klt Calabi-Yau variety.
\end{conjecture}

For $n=2$, the hypersurface $X$ is
$$\{x_0^2+x_1^3x_3+x_2^7x_1+x_3^9x_2=0\}\subset \mathbb{P}^3(5,3,1,1),$$
with an action of the cyclic group $G$ of order $19$ that is free in codimension 1 and faithful
on $H^0(X,K_X)$. So $X/G$ is a klt Calabi-Yau surface of index $19$,
which is the largest possible, by Proposition \ref{dim2}.
For $n = 3$, the hypersurface $X$ is 
$$\{x_0^2 x_4 + x_1^3 x_3 + x_2^5 x_1 + x_3^{19} x_0 + x_4^{32} x_2 = 0\} \subset \mathbb{P}^4(18,12,5,1,1).$$ 
The cyclic group of order $493$ acts on $X$ with quotient a klt Calabi-Yau 3-fold of index $493$. That is well above the index $66$
of the terminal Calabi-Yau 3-fold in Corollary \ref{can-term-ex}.

For $n = 4$, the hypersurface $X$ is 
$$\{x_0^2+x_1^3x_5+x_2^7x_4+x_3^{37}x_2+x_4^{1583}x_1+x_5^{2319}x_3 = 0\} \subset \mathbb{P}^5(1187,791,339,55,1,1).$$  
The cyclic group of order $1201495$ acts on $X$ with quotient a klt Calabi-Yau 4-fold of index $1201495$.  In dimensions $3$ and $4$, our example has the largest
index among all quotients by toric automorphisms of quasi-smooth Calabi-Yau hypersurfaces defined by potentials.  This was verified by computer search, using the databases of Calabi-Yau threefold and fourfold hypersurfaces in \cite{database}.

\end{document}